\documentclass[reqno, a4paper,11pt]{amsart}

\usepackage{amssymb}
\usepackage{amsfonts}
\usepackage{amsmath}

\usepackage{graphicx}
\usepackage[all]{xy}
\usepackage{array}

\usepackage{amsthm}
\newtheorem{theorem}{Theorem}[section]
\newtheorem{lemma}[theorem]{Lemma}
\newtheorem{proposition}[theorem]{Proposition}
\newtheorem{corollary}[theorem]{Corollary}

\theoremstyle{definition}
\newtheorem{definition}[theorem]{Definition}
\newtheorem{ex}[theorem]{Example}
\newtheorem{remark}[theorem]{Remark}

\newtheorem{example}[theorem]{Example}
\newtheorem{examples}[theorem]{Examples}

\theoremstyle{remark}

\newcommand{\Si}{\mathfrak{S}}
\renewcommand{\hom}{\mathrm{hom}}
\newcommand{\Hom}{{\mathrm{Hom}}}
\newcommand{\eEnd}{{\mathrm{end}}}
\newcommand{\End}{{\mathrm{End}}}
\newcommand{\Id}{{\mathrm{Id}}}
\renewcommand{\k}{\Bbbk}

\newcommand{\Fct}{\mathrm{Fct}}
\newcommand{\op}{\mathrm{op}}
\newcommand{\ev}{\mathrm{ev}}
\newcommand{\I}{\mathrm{(I)}}
\newcommand{\II}{\mathrm{(II)}}
\newcommand{\M}{\texttt{M}}
\newcommand{\Q}{\texttt{Q}}

\newcommand{\A}{\mathcal{A}}
\newcommand{\B}{\mathcal{B}}
\newcommand{\C}{\mathcal{C}}
\newcommand{\E}{\mathcal{E}}
\newcommand{\U}{\mathcal{U}}
\newcommand{\W}{\mathcal{W}}
\newcommand{\Qc}{\mathcal{Q}}
\newcommand{\Sc}{\mathcal{S}}

\newcommand{\equi}{\stackrel{\sim}{\longrightarrow}}

\newcommand{\Z}{\mathbb{Z}}

\newcommand{\V}{\mathcal{V}}
\newcommand{\eV}{\mathfrak{V}}

\newcommand{\smod}{\mathsf{smod}}
\newcommand{\svec}{\mathsf{svec}}
\newcommand{\esmod}{\mathfrak{smod}}
\newcommand{\esvec}{\mathfrak{svec}}
\newcommand{\cosmod}{\mathsf{cosmod}}

\newcommand{\vvec}{\mathsf{vec}}
\newcommand{\evec}{\mathfrak{vec}}

\newcommand{\Pol}{\mathsf{Pol}}
\newcommand{\ePol}{\mathfrak{Pol}}
\newcommand{\salg}{\mathfrak{salg}}
\newcommand{\ssch}{\mathfrak{ssch}}
\newcommand{\pol}{\mathfrak{pol}}
\newcommand{\sdim}{\mathrm{sdim}}

\newcommand{\la}{\langle}
\newcommand{\ra}{\rangle}

\title[Spin polynomial functors]{Spin polynomial functors and representations of Schur superalgebras}
\author[J.~Axtell]{Jonathan~Axtell }
\address{Department of Mathematics, Seoul National University, 599 Gwanak-ro, Gwanak-gu, Seoul 151-747, Korea}
\email{jdaxtell@snu.ac.kr}
\thanks{This work was supported by NRF grant \#2011-0027952 and NRF grant \#2012-005700}
\date{\today}
\begin{document}

\sloppy

\maketitle
\begin{abstract}
We introduce categories of homogeneous strict polynomial functors, 
$\Pol^\I_{d,\k}$ and $\Pol^\II_{d,\k}$, defined on vector superspaces  
over a field $\k$ of characteristic not equal 2.  These categories are 
related to polynomial representations of the supergroups $GL(m|n)$ 
and $Q(n)$, respectively.  In particular, we prove an equivalence 
between $\Pol^\I_{d,\k}$, $\Pol^\II_{d,\k}$ and 
the category of finite dimensional supermodules over the Schur 
superalgebra $\Sc(m|n,d)$, $\Qc(n,d)$ respectively 
provided $m,n \ge d$.  We also discuss some aspects of Sergeev 
duality from the viewpoint of the category $\Pol^\II_{d,\k}$.
\end{abstract}
\medskip

\section{Introduction}
Strict polynomial functors were introduced by Friedlander and Suslin in \cite{FS} as 
a tool for use in their investigation of rational cohomology of finite group schemes 
over a field.  Let us briefly recall the definition.

Suppose $\k$ is an arbitrary field, and let $\vvec_\k$ denote the category of 
finite dimensional $\k$-vector spaces.  Also, let $\mathfrak{sch}_\k$ be the category 
of all schemes over $\k$.  Then, by identifying each hom-space with its associated 
affine scheme, we obtain an $\mathfrak{sch}_\k$-enriched category 
$\evec_\k$ (in the sense of \cite{Kelly}) with the same objects 
as $\vvec_\k$.  Although stated somewhat differently in \cite[Definition 2.1]{FS}, 
a {\em strict polynomial functor} may be defined as an 
$\mathfrak{sch}_\k$-enriched functor from $\evec_\k$ to itself. 
From this perspective, it is clear that a strict polynomial functor $T$ yields, 
by evaluation at any $V\in \vvec_\k$, a polynomial representation $T(V)$ 
of the affine group scheme $GL(V)$.  Let us denote by $pol_d(GL(V))$ the category 
of finite  dimensional polynomial representations of $GL(V)$ which are homogeneous 
of degree $d$.  Then a strict polynomial functor $T$ is said to be 
{\em homogeneous of degree $d$} if $T(V) \in pol_d(GL(V))$ 
for all $V\in \evec_\k$.  
We denote by $\mathcal{P}_d$ the category of all such homogeneous strict 
polynomial functors.  The morphisms in $\mathcal{P}_d$ are  
$\mathfrak{sch}_\k$-enriched natural transformations.

Assume that $n\ge d$.  Then, evaluation at $V=\k^n$ in fact gives an equivalence 
of categories $$\mathcal{P}_d\ \equi\ pol_d(GL_n).$$  This follows from the definition 
of the Schur algebra $S(n,d)$ in terms of the coordinate ring of $GL_n$ 
(as in Green's monograph \cite{Green})  and \cite[Theorem 3.2]{FS}, which 
provides an equivalence between $\mathcal{P}_d$ and the category of finite 
dimensional modules over $S(n,d)$.  
We remark that there is an alternate definition of the category $\mathcal{P}_d$ 
which makes the relationship with $S(n,d)$-modules more transparent 
(see e.g. \cite{Krause,P}).  
In this new definition, $\mathfrak{sch}_\k$-enriched functors are replaced by 
$\k$-linear functors defined on a category of {\em divided powers}.

The aim of this paper is to provide an analogue of \cite[Theorem 3.2]{FS} for 
Schur superalgebras.  More specifically, suppose now that $\k$ is a field of 
characteristic $p\neq 2$.  In this context, the 
Schur superalgebras $\mathcal{S}(m|n,d)$ and $\Qc(n,d)$ were 
studied by Donkin \cite{Donkin} and Brundan and Kleshchev 
\cite{BKprojective}, respectively.  In both works there was obtained a 
classification of finite dimensional irreducible supermodules over the 
corresponding Schur superalgebra.  (In \cite{BKprojective} the field $\k$ is
assumed to be algebraically closed.)  In this paper, we introduce categories of 
strict polynomial functors defined on vector superspaces, and we show that 
each such category is equivalent to the category of finite dimensional 
supermodules over one of the above Schur superalgebras. To define strict 
polynomial functors on superspaces, it is more convenient for us to follow the 
approach involving categories of divided powers.  In the last section, however, 
we provide a definition of strict polynomial functors as ``enriched functors" 
which is closer to Friedlander and Suslin's original definition.

The contents of the paper are as follows.  In Section \ref{sec:salg}, we give 
necessary preliminary results concerning superalgebras and supermodules.  
In Section \ref{sec:strictpoly}, we introduce the categories 
$\Pol^{(\dagger)}_{d} = \Pol^{(\dagger)}_{d,\k}$ ($\dagger= \mathrm{I, II}$)
 of homogeneous strict polynomial functors, whose objects are $\k$-linear 
functors defined on categories of vector superspaces.  We also discuss some 
of the usual facets of polynomial functors such as Kuhn duality and Yoneda's 
lemma in this new context. (See \cite{Krause, P, TouzeRingel} for descriptions 
of the corresponding classical notions).

In Section \ref{sec:schursalg}, we prove our main result, Theorem \ref{thm:main1}, 
which gives an equivalence between $\Pol^\I_{d}$, $\Pol^\II_{d}$ and the 
category of finite dimensional supermodules over $\mathcal{S}(m|n,d)$, 
$\Qc(n,d)$ respectively for $m,n \ge d$.  We are then able to obtain in a
classification of irreducible objects in both categories using the classifications of
\cite{Donkin} and \cite{BKprojective}.  
As another application of Theorem \ref{thm:main1}, we give an exact functor from 
the category $\Pol^\II_d$ to the category
of finite dimensional left supermodules over the Sergeev superalgebra $\W(d)$.  
This functor may be viewed as a categorical analogue of  {\em Sergeev duality}, as 
described by Sergeev in \cite{Sergeev} when $p=0$ and by Brundan and 
Kleshchev \cite{BKprojective} in arbitrary characteristic.  Since the representation 
theory of $\W(d)$ is closely related to that of the spin symmetric group algebra 
$\k^-\Si_d$ (c.f. \cite{BKprojective}), we may refer to objects of $\Pol^\II_d$ as 
{\em spin polynomial functors}.  

In Section \ref{sec:enrich}, we conclude by  describing categories $\ePol^\I_d$ and 
$\ePol^\II_d$ consisting of homogeneous {\em $\ssch_\k$-enriched functors}, 
where $\ssch_\k$ denotes the category of all superschems over $\k$.  This definition 
may be viewed as a ``super analogue" of Friendlander and Suslin's original definition 
of strict polynomial functors.  In Theorem \ref{thm:main3} we show that our two 
definitions of strict polynomial functors are equivalent.  One of the benefits of the 
classical approach is that the relationship between strict polynomial functors and 
polynomial representations of the supergroups $GL(m|n)$ and $Q(n)$ appears 
naturally from the definition of $\ssch_\k$-enriched functors.

Finally, let us mention our original motivation for considering categories of 
polynomial functors defined on vector superspaces.  In \cite{HTY}, J. Hong, A. Touz\'e 
and O. Yacobi showed that the  category of all classical polynomial functors 
$$\mathcal{P}=\bigoplus_{d\ge 0} \mathcal{P}_{d},$$ defined over an infinite field $\k$ 
of characteristic $p$, provides a categorification of level 1 Fock space representations 
(in the sense of Chuang and Rouquier) for an affine Kac-Moody algebra 
$\mathfrak{g}$ of type $A_\infty$ (if $p=0$) or of type $A^{(1)}_{p-1}$ (in case $p>0$).
We conjecture that the category of all spin polynomial functors 
$$\Pol^\II_\k= \bigoplus_{d\ge 0} \Pol^\II_{d,\k}$$ defined over an algebraically closed 
field $\k$ of characteristic $p \neq 2$ provides a categorification of certain level 1 Fock 
spaces for an affine Kac-Moody 
algebra of type $B_\infty$ (if $p=0$) or of type $A_{p-1}^{(2)}$ (if $p>2$).

\subsection*{Acknowledgements}
The author wishes to thank Masaki Kashiwara and Myungho Kim for many helpful 
conversations  and suggestions.  
\medskip

\section{Superalgebras and supermodules}\label{sec:salg}
In this section, we give preliminary results on superalgebras and supermodules 
needed for the remainder.  See \cite{BKprojective}, \cite[Ch.12-13]{Kleshchev}, 
\cite[Ch.1]{Leites} and \cite[Ch.3]{Manin} for more details. Although our notation 
sometimes differs from these references.

\subsection{Preliminaries}
Let us fix a field $\k$, which we assume is of characteristic $p\neq 2$.  
A {\em vector superspace} is a $\Z_2$-graded $\k$-vector space 
$M = M_{0}\oplus M_{1}$.  We denote the degree of a homogeneous 
vector, $v \in M$, by $|v| \in \Z_2 $.  A {\em subsuperspace} of 
$M$ is a subspace $N$ of $M$ such that $N = (N\cap M_{0})\oplus 
(N\cap M_{1})$.  We let $\overline{M}$ denote the underlying ordinary 
vector space of a given superspace $M$, and we write 
$\sdim(M) = (m, n)$ if ${\rm dim}(\overline{M_0}) = m$ and 
${\rm dim}(\overline{M_1}) = n$.  

Given a pair of vector superspaces $M, N$ we view the direct sum 
$M\oplus N$ and the tensor product $M\otimes N$ as vector 
superspaces by setting: $(M\oplus N)_i = M_i \oplus N_i$ 
($i\in \Z_2$), $(M\otimes N)_{0} = M_{0}\otimes N_{0} \oplus 
M_{1}\otimes N_{1}$ and $(M\otimes N)_{1} = M_{0}\otimes 
N_{1} \oplus M_{1}\otimes N_{0}$.  We also consider the vector 
space 
$\Hom(M,N) = \Hom_{\k}(M,N)$ of all $\k$-linear maps of $M$ into $N$ 
as a superspace by letting $\Hom(M,N)_i$ consist of the 
{\em homogeneous maps of degree $i$} for $i \in \Z_2$, i.e. the 
maps $f:M \rightarrow N$ such that $f_i(M_j) \subseteq N_{i+j}$ for 
$j \in \Z_2$.  The elements of $\Hom(M,N)_{0}$ are called 
{\em even} linear maps, and the elements of 
$\Hom(M,N)_{1}$ are called {\em odd}.   The $\k$-linear 
dual $M^\vee = \Hom(M,\k)$ is a superspace by viewing 
$\k$ as vector superspace concentrated in degree 0.  
Let $\svec_\k$ denote the category of all finite dimensional 
$\k$-vector superspaces  with arbitrary linear maps as
morphisms.

If $M\in \svec_\k$, then for $f\in M^\vee$ and $v\in M$, 
we write 
$$\la f,v\ra\ =\ f(v) \in \k$$ 
to denote the pairing between $M$ and $M^\vee$.
We identify $M$ with $(M^\vee)^\vee$ as superspaces
 by setting 
 \begin{equation}\label{eq:dual}
 \la v, f \ra\ =\  (-1)^{|v| |f|} \la f, v \ra
\end{equation}
for $v\in M, f\in M^\vee$.



A {\em superalgebra} is a  superspace $\mathcal{A}$ with the 
additional structure of an associative unital $\k$-algebra such 
that $\mathcal{A}_i \mathcal{A}_j \subseteq \mathcal{A}_{i+j}$ 
for $i, j \in \Z_2$.  By forgetting the grading we may consider 
any superalgebra $\mathcal{A}$ as an ordinary algebra, 
denoted by $\overline{\mathcal{A}}$.  A {\em superalgebra 
homomorphism} $\vartheta: \mathcal{A} \rightarrow 
\mathcal{B}$ is an even linear map that is an algebra 
homomorphism in the usual sense; its kernel is a 
{\em superideal}, i.e., an ordinary two-sided ideal 
which is also a subsuperspace.  An {\em antiautomorphism} 
$\tau: \mathcal{A} \rightarrow \mathcal{A}$ of a superalgebra 
$\mathcal{A}$ is an even linear map which satisfies 
$\tau(a b) = \tau(b) \tau(a)$. 

Given two superalgebras $\mathcal{A}$ and $\mathcal{B}$, 
we view the tensor product of superspaces 
$\mathcal{A}\otimes \mathcal{B}$ as a superalgebra with 
multiplication defined by
\begin{equation*}
(a\otimes b)(a' \otimes b') = (-1)^{|b| |a'|} (a a') \otimes (b b') \qquad (a,a' \in \mathcal{A}, b,b' \in \mathcal{B}).
\end{equation*}
We note that 
$\mathcal{A}\otimes \mathcal{B} \cong \mathcal{B} \otimes \mathcal{A}$, 
an isomorphism being given by 
\begin{equation*}
 a\otimes b \mapsto (-1)^{|a| |b|} b\otimes a \quad (a \in \mathcal{A}, b \in \mathcal{B}).
\end{equation*}

\subsection{Tensor powers} 
Let $M$ be a vector superspace. The {\em tensor superalgebra} $T^*M$ is 
the tensor algebra $$T^*M = \bigoplus_{d\ge 0} M^{\otimes d}$$ regarded 
as a vector superspace.  It is the free associative ($\Z$-graded) 
superalgebra generated by $M$.

The {\em symmetric superalgebra} $S^*M$ is the quotient of $T^*M$ by 
the super ideal 
$$\mathcal{I}= \langle x\otimes y - (-1)^{|x| |y|}y\otimes x ;\ x,y \in M \rangle.$$ 
Since $\mathcal{I}$ is a $\Z$-graded homogeneous ideal, there exists a gradation 
$S^*M = \bigoplus_{d \ge 0} S^d M$.
Now we may view the ordinary symmetric algebra $Sym^*\overline{M_0}$ as a 
superspace 
concentrated in degree zero.  We may also view the ordinary exterior algebra 
$\Lambda^*\overline{M_1}$ as a superspace by reducing its $\Z$-grading mod $2\Z$.  
In this way both $Sym^*\overline{M_0}$ and $\Lambda^*\overline{M_1}$ may be 
regarded as $\Z$-graded superalgebras.
One may check that we have a $\Z$-graded superalgebra isomorphism:
\begin{equation} \label{eq:sym1}
S^*M \ \cong \  Sym^*\overline{M_0}\otimes  \Lambda^*\overline{M_1}.
\end{equation}

A superalgebra $\mathcal{A}$ is called {\em commutative} if $ab = (-1)^{|a| |b|} ba$ for all 
$a,b \in \mathcal{A}$.   The superalgebra $S^*M$ is the free commutative ($\Z$-graded) 
superalgebra generated by $M$.

\subsection{Divided powers}\label{sec:divided}
There is a unique (even) right action of the symmetric group $\mathfrak{S}_d$ on the tensor 
power $M^{\otimes d}$ such that each transposition $(i\ i+1)$ for $1\le i \le d-1$ acts by: 
$ (v_1 \otimes \cdots \otimes v_d). (i\ i+1) \ = $
\[
\ (-1)^{|v_i| |v_{i+1}|} v_1\otimes \cdots \otimes v_{i+1} \otimes v_i \otimes \cdots \otimes v_d,
\]
for any $v_1, \dots, v_d \in M$ with $v_i, v_{i+1}$ $\Z_2$-homogeneous.  Denote the invariant 
subsuperspace of this action by
$$ \Gamma^d M:= (M^{\otimes d})^{\mathfrak{S}_d}.$$
Now the symmetric power is the coinvariant superspace 
$S^dM = (M^{\otimes d})_{\Si_d}$.
Hence, given arbitrary vector superspaces $V,W$ there are natural {\em even}
isomorphisms
\begin{equation*}
\Hom_{\Si_d}(V,M^{\otimes d}) \cong \Hom(V, \Gamma^dM),\ \   
\Hom_{\Si_d}(M^{\otimes d}, W) \cong \Hom(S^dM, W),
\end{equation*}
where $V$ and $W$ are considered as trivial $\Si_d$-modules.
There is also a right action of $\Si_d$ on $(M^{\otimes d})^\vee$ given by 
$(f.\sigma)(v) = f(v.\sigma^{-1})$, for $f\in (M^{\otimes d})^\vee, v\in M^{\otimes d}$ 
and $\sigma \in \Si_d$.  Furthermore, there is a natural even isomorphism
\begin{equation}\label{eq:duality1}
\Gamma^d(M)^\vee\ \cong S^d(M^\vee).
\end{equation}

Now let $\Gamma^*M$ be the $\Z$-graded superspace $\bigoplus_{d\ge 0} \Gamma^d M$.
Also let $D^*\overline{M_0}$ denote the ordinary divided powers algebra of the vector space 
$\overline{M_0}$ (cf. \cite{Bourbaki}).  Viewed as a vector superspace concentrated in degree 
zero,  $D^*\overline{M_0}$ is a $\Z$-graded superalgebra.  Also note that we have a natural 
embedding of superspaces: $\Lambda^d \overline{M_1} \hookrightarrow (M_1)^{\otimes d}$.
We then have an even isomorphism of $\Z$-graded superspaces
\begin{equation}\label{eq:div1}
\Gamma^*(M_0) \cong D^*\overline{M_0} \otimes \Lambda^*\overline{M_1}.
\end{equation} 
The isomorphism (\ref{eq:div1}) defines a superalgebra structure on $\Gamma^*M$ which we 
call the  {\em divided power superalgebra}.

\subsection{Supermodules}
Let $\mathcal{A}$ be a superalgebra.  A {\em left $\mathcal{A}$-supermodule} is a superspace $V$ 
which is a left $\mathcal{A}$-module in the usual sense, such that 
$\mathcal{A}_i V_j \subseteq \mathcal{A}_{i+j}$ for $i,j \in \Z_2$.   One may similarly define 
{\em right $\mathcal{A}$-supermodules}.  A {\em homomorphism} $\varphi: V \rightarrow W$ of left 
$\mathcal{A}$-supermodules $V$ and $W$ is a (not necessarily homogeneous) linear map such that 
\[
\varphi(av) =(-1)^{|\varphi| |a|} a \varphi(v) \qquad (a \in \mathcal{A}, v \in V).
\]
We denote by $_\mathcal{A}\smod$ the category of finite dimensional left $\mathcal{A}$-supermodules 
with $\mathcal{A}$-homomorhpisms.  
A {\em homomorphism}, $\varphi: V \rightarrow W$, of right $\mathcal{A}$-supermodules $V$ and $W$ 
is a (not necessarily homogeneous) linear map such that
\[
\varphi(v a) = \varphi(v) a \qquad (a \in \mathcal{A}, v \in V).
\]
Let $\smod_{\mathcal{A}}$ denote the category of finite dimensional right $\mathcal{A}$-supermodules 
with $\mathcal{A}$-homomorphisms.  

\subsection{Parity change functor}
Suppose $V$ is a left or right $\mathcal{A}$-supermodule.  Then define a new supermodule $\Pi V$ which 
is the same vector space as $V$ but with opposite $\Z_2$-grading.  For right supermodules, the new right 
action is the same as in $V$.  For left supermodules, the new left action of $a\in \mathcal{A}$ on 
$v\in \Pi V$ is defined in terms of the old one by $a\cdot v := (-1)^{|a|}av$.  On a morphism $f$, $\Pi f$ is the 
same underlying linear map as $f$.  
Let us write $\k^{m|n} = \k^m \oplus (\Pi \k)^n$.

 \begin{examples}\label{ex:findimsalg}  We have the following examples of finite dimensional associative 
 superalgebras.
\smallskip
\begin{itemize}
\item[(i)]  If $M$ is a superspace, then $\End(M) = \Hom_\k(M,M)$ is a superalgebra.  In particular, we 
write $\mathcal{M}_{m,n} = \End(\k^{m|n})$.   
\smallskip
\item[(ii)]  Let $V\in \svec_\k$, and suppose $J$ is a degree one involution in $\End(V)$.  This is possible 
if and only if $\dim \overline{V_0} = \dim \overline{V_1}$.  Let us consider the superalgebra
$$ \Qc(V,J) = \{ \varphi \in \End(V);\ \varphi J = (-1)^{|\varphi|} J \varphi \}.$$
Suppose that $\sdim V = (n,n)$, and let $\{v_1, \dots, v_n\}$ (resp. $\{v_1', \dots, v_n'\}$) a basis of $V_0$ 
(resp. $V_1$).  Let $J_V$ be the unique involution in $\End_\k(V)$ such that $J v_i = v_i'$ for $1\le i \le n$.  
Then we may write elements of $\Qc(V,J_V)$ with respect to the basis 
$\{v_1, \dots, v_n, v_1',\dots, v_n'\}$ as matrices of the form
\begin{equation}\label{eq:typeQ}\left( \begin{array}{cc}
A& B \\
-B & A 
\end{array} \right), \end{equation} where $A,B$ are $n\times n$ matrices, with $A=0$ for odd endomorphisms 
and $B=0$ for even ones.  

Suppose that $\k$ is algebraically closed.  Recall (cf. \cite[ch.12]{Kleshchev}) that all odd involutions $J \in \End(V)$
  are then
mutually conjugate (by an invertible element of $\End(V)_0$).  Hence, any superalgebra $\Qc(V,J)$ 
is isomorphic to the superalgebra $\Qc_n$, consisting of all matrices of the form (\ref{eq:typeQ}).
\smallskip
\item[(iii)]  The {\em Clifford superalgebra}, $\C(d)$, is the superalgebra generated by odd elements 
$c_1, \dots, c_d$ subject to the relations $c_i^2=1$ for $i=1, \dots, d$ and $c_i c_j = -c_j c_i$ for all $i\neq j$.  
There is an isomorphism
\[
\C(d_1+d_2) \simeq \C(d_1) \otimes \C(d_2),
\]
defined by mapping $c_i \mapsto c_i\otimes 1$ and $c_{d_1+j} \mapsto 1\otimes c_j$, 
for $1\leq i\leq d_1$ and $1\leq j \leq d_2$.  Hence, we have
\begin{equation}\label{eq:tensor}
\C(d) \simeq \C(1) \otimes \cdots \otimes \C(1) \quad (d\ {\rm copies}).
\end{equation}
\end{itemize}
\end{examples}

\subsection{Categories enriched over $\svec_\k$}
We say a category $\V$ is an {\em $\svec_\k$-enriched category} 
if the hom-sets $\hom_\V(V,W)$  ($V,W\in \V$) are finite 
dimensional $\k$-superspaces while composition is bilinear and 
{\em even}.  I.e., if $U,V,W \in \V$, then composition induces an even 
linear map:
$${\rm hom}_{\V}(V,W) \otimes {\rm hom}_{\V}(U,V) \rightarrow {\rm hom}_{\V}(U,W).$$
We will write $$V \cong W,$$ if  $V,W$ are isomorphic in $\V$.  If there 
is an even isomorphism $\varphi:V \cong W$ (i.e., 
$\varphi\in {\rm hom}_\V(V,W)_0$), we use the notation 
$$V \simeq W.$$
Let $\V_\ev$ denote the subcategory of $\V$ consisting of the same 
objects but only even homomorphisms.

For a superalgebra $\mathcal{A}$, the categories $_\A \smod$ and $\smod_\A$ are 
naturally $\svec_\k$-enriched categories.  Furthermore, the subcategories 
$(_\A \smod)_{\rm ev}$ and $(\smod_\mathcal{A})_{\rm ev}$ are abelian 
categories in the usual sense.  This allows us to make use of the basic notions of 
homological algebra by restricting our attention to only even morphisms.  For 
example, by a short exact sequence in $_\A\smod$ 
(resp. $\smod_\A$), we mean a sequence
\[
0 \rightarrow V_1 \rightarrow V_2 \rightarrow V_3 \rightarrow 0,
\]
with all the maps being {\em even}.  All functors between the $\svec_\k$-enriched 
categories which we consider will send even morphisms to even morphisms.  So they 
will give rise to the corresponding functors between the underlying even 
subcategories.

Now if $\V$ is an $\svec_\k$-enriched category, let $\V^-$ denote the category with 
the same objects and morphisms as $\V$ but with modified composition law: 
$\varphi\circ \varphi' = (-1)^{|\varphi| |\varphi'|} \varphi \varphi'$, where 
$\varphi \varphi'$ denotes composition in $\V$.  

\begin{example}\label{ex:k minus}
If $f\in \Hom(M,N)$ for some $M,N \in \svec_\k$, we let $f^-:M\rightarrow N$
be the linear operator defined by
$$f^-(v)= (-1)^{|f| |v|} f(v).$$
It can be checked that mapping $M\mapsto M$ and $f\mapsto f^-$
for all $M,N\in \svec_\k$, $f\in \Hom(M,N)$ gives an equivalence
\begin{equation}\label{eq:k minus}
 (\svec_\k)^- \equi \svec_\k.
 \end{equation}
\end{example}

Given a superaglebra $\A$, also define a new superalgebra $\A^-$, with the 
same elements as $\A$ and modified multiplication law 
$a\cdot b = (-1)^{|a| |b|}ab$.  Notice that for any $V\in \V$, the superspace 
$\eEnd_\V(V) = \hom_\V(V,V)$ is a superalgebra and
\begin{equation}\label{eq:minus1}
\eEnd_{\V^-}(V) = \eEnd_\V(V)^-.
\end{equation}

\subsection{Schur's lemma}
It is possible that an irreducible $\A$-supermodule may become reducible 
when considered as an $\overline{\A}$-module.  We say that an 
irreducible left $\mathcal{A}$-supermodule $V$ is of {\em type $\M$} if the left 
$\overline{\A}$-module $\overline{V}$ is irreducible, and otherwise we say 
that $V$ is of {\em type $\Q$}.  We have the following criterion.

\begin{lemma}[Schur's lemma]\label{lem:Schur}
Suppose $\A$ is a superalgebra, and let $V$ be a finite dimensional irreducible left 
$\mathcal{A}$-supermodule.  Then 
$$\dim \overline{\End_\A(V)} = \begin{cases}
1 & \mbox{if $V$ is of type $\M$},\\
2 & \mbox{if $V$ is of type $\Q$}.
\end{cases}$$
\end{lemma}

\begin{example}
The superspace $\k^{m|n}$ is naturally an irreducible left $\mathcal{M}_{m,n}$-supermodule 
of type $\M$.  On the other hand, the superspace $V= \k^{n|n}$ is naturally an irreducible left 
$\Qc_n$-supermodule.  Since $\dim \overline{\End_{\Qc_n}(V)} > 1$, it 
follows that $V$ is of type $\Q$.  This explains the given names for the types.
\end{example}

\subsection{Wedderburn's theorem}
If $V,W \in{}_\mathcal{A}\smod$ (resp. $\smod_\mathcal{A}$), we let $\Hom_\mathcal{A}(V,W)$ 
denote the set of $\mathcal{A}$-homomorphisms from $V$ to $W$.  Also let $\End_\A(V)$ 
denote the superalgebra of all $\A$-supermodule endomorphisms of $V$.  Given a finite 
dimensional superalgebra $\mathcal{A}$ and some $V \in{}_\mathcal{A}\smod$ 
(resp. $\smod_\mathcal{A}$), we have a natural isomorphism
\begin{equation}\label{eq:superalg}
{\rm Hom}_{\mathcal{A}}(\mathcal{A}, V) \simeq V
\end{equation}
of vector superspaces.

Let $\mathcal{A}$ be a superalgebra.  A {\em subsupermodule} of a left (resp. right) 
$\mathcal{A}$-supermodule is a left (resp. right) $\mathcal{A}$-submodule, in the 
usual sense, which is also a subsuperspace.  A left (resp. right) 
$\mathcal{A}$-supermodule is {\em irreducible} if it is non-zero and has no 
non-zero proper subsupermodules.  We say that a left (resp. right) 
$\mathcal{A}$-supermodule is {\em completely reducible} if it can be decomposed 
as a direct sum of irreducible subsupermodules.  Call $\mathcal{A}$ {\em simple} if 
$\mathcal{A}$ has no non-trivial superideals, and a {\em semisimple superalgebra} 
if $\mathcal{A}$ is completely reducible viewed as a left $\mathcal{A}$-supermodule.  
Equivalently, $\mathcal{A}$ is semisimple if {\em every} left 
$\mathcal{A}$-supermodule is completely reducible.  We have:

\begin{theorem} \label{thm:Wedb}
Let $\mathcal{A}$ be a finite dimensional superalgebra.  The following are 
equivalent:
\begin{itemize}
\item[\rm(i)] $\mathcal{A}$ is semisimple;
\item[\rm(ii)] every left (resp. right) $\mathcal{A}$-supermodule is completely reducible;
\item[\rm(iii)] $\mathcal{A}$ is a direct product of finitely many simple 
superalgebras.
\end{itemize}
\end{theorem}
\smallskip
\begin{ex}\label{ex:cliff2}
The Clifford superalgebra $\C(1)$ may be realized as the 
superalgebra of $2\times2$ matrices of the form $\left\{ 
 \left( \begin{array}{cc}
a& b \\
b& a 
\end{array} \right)\ \Big|\ a,b \in \k \right\}$.  
The generator $c_1$ of $\C(1)$ corresponds to the matrix 
$J_1= \left( \begin{array}{cc}
0& 1 \\
1& 0 
\end{array} \right)$.  
One may check that $\C(1)$ is a simple superalgebra with 
a unique right (resp. left) irreducible supermodule up to isomorphism.  
In fact, $\C(1)$ is an irreducible supermodule over itself 
with respect to right (resp. left) multiplication, and we denote this 
supermodule by $ \U_r(1)$ (resp. $\U_l(1)$). In 
the sequel, we usually write $\U(1)=\U_r(1)$.

Suppose that $V\in {}_{\C(1)}\smod$, $V'\in \smod_{\C(1)}$.
By Theorem \ref{thm:Wedb}, 
there exist decompositions $V \simeq \U_l(1)^{\oplus n}$ 
and $V'\simeq \U(1)^{\oplus n'}$, for some 
$n, n' \in \Z_{\ge 0}$.
Hence, we have $\sdim(V)=(n,n)$ and $\sdim(V')=(n',n')$, and 
there exists a basis of $V$ (resp. $V'$) such that $c_1 \in \C(1)$ acts on 
$V$ (resp. $V'$) via multiplication by the matrix 
\begin{equation}\label{eq:c1gen}
 \left( \begin{array}{cc}
0 & I_N \\
I_N& 0 
\end{array} \right),
\end{equation}
where $I_N$ is the $N\times N$ unit matrix for $N=n,n'$ respectively.

Now let $V,W \in  {}_{\C(1)}\smod$ (resp. $\smod_{\C(1)}$).  
As mentioned above, we may assume that $\sdim(V) =(m,m)$ (resp. 
$\sdim(W)=(n,n))$ for some $m,n \in \Z_{\ge 0}$.  
By equation (\ref{eq:c1gen}), we may choose respective bases of $V$ and 
$W$ such that ${\rm Hom}_{\C(1)}(V,W)$ consists of all matrices of 
the form
\begin{equation}\label{eq:C1homs}
 \left( \begin{array}{cc}
A & B \\
-B& A 
\end{array} \right) \quad \left(\mbox{resp. }
 \left( \begin{array}{cc}
A & B \\
B& A 
\end{array} \right) \right),
\end{equation}
where $A, B$ are $n\times m$ matrices in $\k$, and $A=0$ (resp. $B=0$) for 
odd (resp. even) homomorphisms.
\end{ex}

\begin{remark}\label{rem:cliff1}
Notice that $\overline{\C(1)}$ is commutative as an ordinary algebra
even though $\C(1)$ is not a commutative superalgebra.  Hence, the
objects of $_{\C(1)}\smod$ can be identified with the objects of 
$\smod_{\C(1)}$.  It can be checked using (\ref{eq:C1homs}) that we 
have an equivalence
\begin{equation}\label{eq:minus2} 
(_{\C(1)}\smod)^- \equi \smod_{\C(1)},
\end{equation}
given by mapping $V\mapsto V$ and $\varphi\mapsto \varphi^-$ for all
$V,W\in ({}_{\C(1)}\smod)^-$ and $\varphi\in \Hom_{\C(1)}(V,W)$.
\end{remark}

\begin{remark}\label{rem:cliff2}
Suppose that $V \in {}_{\mathcal{C}(1)}\smod$ and $\sdim(V)=(n,n)$.  
Then it is clear from (\ref{eq:C1homs}) that we have a superalgebra 
isomorphism $\mathcal{Q}_n \cong \End_{\mathcal{C}(1)}(V)$.  
Now suppose that there is a $\sqrt{-1}\in \k$.
If $V' \in  \smod_{\mathcal{C}(1)}$ and again $\sdim(V')=(n,n)$,
then it is not difficult to check that we also have an isomorphism  
$\mathcal{Q}_n \cong \End_{\mathcal{C}(1)}(V')$ of superalgebras.
\end{remark}

\subsection{Wreath products}
Suppose $\A$ is an associative superalgebra.  Notice that 
the right action of $\sigma \in \Si_d$ on the tensor power 
$\A^{\otimes d}$ is in fact a superalgebra automorphism.    
Denote by $\A \wr \Si_d$ the vector superspace
 \begin{equation*}
 \A\wr\Si_d = \k \Si_d \otimes \A^{\otimes d}
\end{equation*}
(where the group algebra $\k \Si_d$ is viewed as superspace 
concentrated in degree zero).  
We then consider $\A\wr\Si_d$ as a superalgebra with 
multiplication defined by the rule
\[
(\sigma \otimes a)(\sigma' \otimes b) = \sigma \sigma' \otimes (a\cdot w') b
\]
for $\sigma, \sigma' \in \Si_d$, $a,b \in \A$.  
In what follows, we will 
identify $\A^{\otimes d}$ (resp.\ $\k\Si_d$) with the subsuperalgebra 
$1\otimes \A^{\otimes d}$ (resp. $\k \Si_d \otimes 1$) of $\A\wr\Si_d$.

\begin{example}[Sergeev superalgebra]
If $\A = \k$, then $\k \wr \Si_d = \k \Si_d$, the group 
algebra of $\Si_d$.  On the other hand, if we identity $C(d)$
with $\C(1)^{\otimes d}$ via the isomorphism (\ref{eq:tensor}),
then $\C(1) \wr \Si_d = \W(d),$ the {\em Sergeev superalgebra}
(cf. \cite{BKprojective}). 
\end{example}

\subsection{Tensor products of supermodules}\label{subsec:tensor}
Given left supermodules $V$ and $W$ over arbitrary superalgebras 
$\mathcal{A}$ and $\mathcal{B}$ respectively, the tensor 
product $V\otimes W$ is a left $\mathcal{A}\otimes \mathcal{B}$-supermodule 
with action defined by 
$(a\otimes b).(v\otimes w) = (-1)^{|b| |v|} a.v \otimes b.w$, for all 
homogeneous $a \in \mathcal{A}, b\in \mathcal{B}, v\in V, w\in W$.  
(Analogously, if $V$ and $W$ are right supermodules, the action of 
$\mathcal{A}\otimes \mathcal{B}$ on $V\otimes W$ is given by 
$(v\otimes w).(a\otimes b) = (-1)^{|w| |a|} v.a \otimes w.b$, for all 
homogeneous $a \in \mathcal{A}, b\in \mathcal{B}, v\in V, w\in W$.)   
If $\varphi: V\rightarrow V'$ (resp. $\varphi':W \rightarrow W'$) is a 
homogeneous homomorphism of left $\mathcal{A}$- (resp. 
$\mathcal{B}$-) supermodules, then 
$\varphi\boxtimes \varphi': V\otimes W \rightarrow V'\otimes W'$ is a 
homomorphism of left $\mathcal{A}\otimes \mathcal{B}$-supermodules, 
where  
\begin{equation}\label{eq:outer}
 \varphi \boxtimes \varphi' (v\otimes w) = (-1)^{|\varphi'| |v|}\varphi v \otimes \varphi' w.
\end{equation}  
(The previous statement holds also for right supermodules.  I.e., the 
outer tensor product $\varphi\boxtimes \varphi'$ of right supermodule 
homomorphisms, $\varphi:V\rightarrow W$ and 
$\varphi':V'\rightarrow W'$, is given by the same formula (\ref{eq:outer}).)

As a particular example, if $M,M',N,N' \in \svec_\k$, then (\ref{eq:outer}) 
gives a natural isomorphism
\begin{equation}\label{eq:isom3}
\Hom_\k (M, N)\otimes \Hom_\k (M', N')\ \simeq\ \Hom_\k (M\otimes M', N\otimes N'),
\end{equation}
which sends $f \otimes f' \mapsto f \boxtimes f'$.
More generally, we have the following.
\begin{lemma}\label{lem:tensor}
Suppose $\B$ is a simple finite dimensional superalgebra.
If $V,W \in \smod_\B$, then there is a canonical isomorphism
\begin{equation}\label{eq:isom4}
  \Hom_\B(V,W)^{\otimes d}\ \simeq\ \Hom_{\B^{\otimes d}}(V^{\otimes d}, W^{\otimes d}),
  \end{equation}
which maps $f_1 \otimes \cdots \otimes f_d$ onto $f_1 \boxtimes \cdots \boxtimes f_d$.
\end{lemma}
\begin{proof}
It suffices to consider $d=2$.  The map $f\otimes g\mapsto f\boxtimes g$ is clearly injective.  To check that it is 
surjective we may use Lemma \ref{lem:Schur} together with Theorem \ref{thm:Wedb}  and
 \cite[Lemma 12.2.13]{Kleshchev}.
\end{proof}

\smallskip

\section{Strict polynomial functors of types I and II}\label{sec:strictpoly}

 We now introduce the categories $\Pol^\I_{d,\k}$ and 
 $\Pol^\II_{d,\k}$ consisting of homogeneous strict polynomial functors. 
  Such polynomial functors are realized as $\k$-linear functors 
  between an appropriate pair of $\svec_{\k}$-enriched 
  categories.

\subsection{Categories of divided powers}  
Suppose that $\B$ is a simple finite dimensional superalgebra,
and let $\V = \smod_\B$.  We then define a new category 
$\Gamma^d\V$.  The objects of $\Gamma^d\V$ are the same
as those of $\V$, i.e. finite dimensional right $\B$-supermodules.
Given $V,W \in \smod_\B$, set
$$\hom_{\Gamma^d\V}(V,W) := \Gamma^d \Hom_\B(V,W).$$  
In order to define the composition law, we make use of the 
following lemma.

\begin{lemma}
Suppose $V\in \smod_\B$.  Then $V^{\otimes d} \in \smod_{\B\wr\Si_d}$, 
where $\B\wr \Si_d$ is the wreath product defined above.  If $V,W \in \smod_{\B}$, 
then we further have a natural isomorphism
\begin{equation}\label{eq:isom5}
\Hom_{\B\wr\Si_d}(V^{\otimes d},W^{\otimes d})\ \simeq\ \Gamma^d\Hom_{\B}(V,W).
\end{equation}
\end{lemma}

\begin{proof}
By Lemma \ref{lem:tensor}, $V^{\otimes d} \in \smod_{\B}$. 
One may check that for any $\sigma \in \Si_d$, we have
\begin{equation}\label{eq:compatible}
(v.a).\sigma = (v.\sigma).(a\cdot \sigma) \qquad (v \in V^{\otimes d},\ a \in \B),
\end{equation}
where $\Si_d$ acts on $\B^{\otimes d}$ on the right as in the definition of $\B\wr \Si_d$.  
Now given a homomorphism $\varphi \in \Hom_{\B}(V^{\otimes d},W^{\otimes d})$, 
it follows from (\ref{eq:compatible}) that  
$\varphi^{\sigma} \in \Hom_{\B}(V^{\otimes d}, W^{\otimes d})$, 
where $\varphi^\sigma:V^{\otimes d} \rightarrow W^{\otimes d}$ is the linear map defined by 
$\varphi^{\sigma}(v) = (\varphi(v.\sigma^{-1})).\sigma$ for any $v\in V^{\otimes d}$.  
One may then check that 
$$ \Hom_{\B}(V^{\otimes d}, W^{\otimes d})^{\Si_d}\ =\ \Hom_{\B\wr\Si_d}(V^{\otimes d},W^{\otimes d}).$$
It is also not difficult to check that the isomorphism (\ref{eq:isom4}) is in fact an 
isomorphism of $\Si_d$-modules.  Hence we have a canonical isomorphism 
\begin{align*}
\Gamma^d \Hom_{\B}(V,W) =&\ (\Hom_{\B}(V,W)^{\otimes d})^{\Si_d}\\
\simeq&\ \Hom_{\B^{\otimes d}}(V^{\otimes d},W^{\otimes d})^{\Si_d}\ = \Hom_{\B\wr\Si_d}(V^{\otimes d},W^{\otimes d})
\end{align*}
\end{proof}
Using the isomorphism in the previous lemma for any $V,W \in \V = \smod_{\B}$, 
composition in $\smod_{\B\wr\Si_d}$ induces a composition law 
in $\Gamma^d\V$.  As primary examples, we have the categories 
$\Gamma^d_\M = \Gamma^d\svec_\k$ and $\Gamma^d_\Q = \Gamma^d \smod_{\C(1)}$.

\subsection{Schur superalgebras}
Let $M = \k^{m|n}$.  Then we have a superalgebra isomorphism
$$\eEnd_{\Gamma^d_{\M}}(M)\ = \End_{\k \Si_d}(M^{\otimes d}) \cong \ \mathcal{S}(m|n,d),$$
where $\mathcal{S}(m|n,d)$ is the Schur superalgebra defined in \cite{Donkin}.
\medskip

Let $V = \U(1)^{\oplus n} \in \Gamma^d_{\Q}$.  Then we  have another 
isomorphism of superalgebras
$$\eEnd_{\Gamma^d_{\Q}}(V)\ = \End_{\mathcal{W}(d)}(V^{\otimes d}) \cong \ \Qc(n,d),$$
where $\Qc(n,d)$ is the  Schur superalgebra defined in \cite{BKprojective}.

\subsection{Strict polynomial functors}
Notice that $\Gamma^d_\M$ and $\Gamma^d_\Q$ are both 
$\svec_\k$-enriched categories.    
Let $\Pol^\I_d = \Fct_\k(\Gamma^d_\M, \svec_\k)$, the category 
of even $\k$-linear functors from $\Gamma^d_\M$ to $\svec_\k$.  
Similarly, let $\Pol^\II_d= \Fct_\k(\Gamma^d_\Q,\svec_\k)$.  In both cases, 
morphisms are natural transformations between functors, and 
objects of either category are called {\em (homogeneous) strict
polynomial functors}.

Given $S,T \in \Pol^\I_d$ (resp. $\Pol^\II_d$), the set of all 
natural transformations $\eta: S\rightarrow T$ is 
naturally a vector superspace.  It this way, we see that 
$\Pol^\I_d$ and $\Pol^\II_d$ are $\svec_\k$-enriched 
categories.  The even subcategories 
$(\Pol^\I_d)_{\rm ev}, (\Pol^\II_d)_\ev$ both inherit the 
structure of an abelian category, since kernels, 
cokernels, products and coproducts can be computed 
in the target category $(\svec_\k)_{\rm ev}$.

\begin{examples}  We have the following examples of 
strict polynomial functors belonging to $\Pol^\I_d$ 
for some $d\ge 1$.
\begin{itemize}
\item[(i)]  The identity functor, 
${\rm Id}: \svec_\k \rightarrow \svec_\k$, belongs to 
$\Pol^\I_1$. Another object of $\Pol^\I_1$ is the 
parity change functor 
$\Pi: \svec_\k \rightarrow \svec_\k$, introduced in 
the previous section.
\smallskip
\item[(ii)]  The functor $\otimes^d \in \Pol^\I_d$ sends 
an object $M\in \Gamma^d_\M$ to $M^{\otimes d}$ 
and a morphism $f \in \hom_{\Gamma^d_\M}(M,N)$ 
to the same underlying map regarded as an element 
of $\Hom_\k(M^{\otimes d}, N^{\otimes d})$.
\smallskip
\item[(iii)]  Given $M\in \svec_\k$, let 
$\Gamma^{d,M} = \hom_{\Gamma^d_\M}(M, -)$, a 
representable functor in $\Pol^\I_d$.  In particular, 
if $M = \k^{m|n}$ we write 
$\Gamma^{d,m|n} = \Gamma^{d,M}$.
\end{itemize}
\end{examples}
Notice that for any $M\in \svec_\k$, we have a 
canonical isomorphism 
\begin{equation}\label{eq:div type I}
\Gamma^d M \simeq \Gamma^{d,1|0}(M),
\end{equation}
 since $\hom_{\Gamma^d_\M}(\k, M) = \Gamma^d\Hom_\k(\k,M) \simeq \Gamma^dM$.
\medskip

Let us identify $\smod_{\C(1)}$ as a subcategory 
of $\svec_\k$.  Since we may view $\k \Si_d$ as 
a subsuperalgebra of $\W(d)$, there is a 
restriction functor from $\smod_{\W(d)}$ to 
$\smod_{\k \Si_d}$.   This in turn yields an even 
$\k$-linear functor, 
$\mathrm{Res}: \Gamma^d_\Q \rightarrow \Gamma^d_\M$, 
which acts as the identity on objects and by 
restriction on morphisms.  Hence, composition 
yields a functor 
$$- \circ \mathrm{Res}: \Pol^\I_d \rightarrow \Pol^\II_d.$$

\begin{examples} The following are examples of 
objects in  $\Pol^\II_d$, for some $d\ge 1$.
\begin{enumerate} 
\item[(i)]  We use the same notation, 
${\rm Id} = {\rm Id} \circ \mathrm{Res}: \smod_{\C(1)} \rightarrow \svec_\k$, 
to denote the restriction of the identity functor.  Clearly 
Id$\in \Pol^\II_1$.  Also, note that we have an even 
isomorphism $$ \Pi\circ \mathrm{Res}\ \simeq {\rm Id}$$ 
in $\Pol^\II_1$.
\smallskip
\item[(ii)] The functor 
$\otimes^d = \otimes^d \circ \mathrm{Res} \in \Pol^\II_d$ 
sends an object $V\in \Gamma^d_\Q$ to $V^{\otimes d}$ 
and a morphism $\varphi \in \hom_{\Gamma^d_\Q}(V,W)$ 
to the same underlying map regarded as an element of 
$\Hom_\k(V^{\otimes d}, W^{\otimes d})$.
\smallskip
\item[(iii)]  If $V \in \Gamma^d_\Q$, let 
$\Gamma^{d,V} = \hom_{\Gamma^d_\Q}(V,-)$, which 
belongs to $\Pol^\I_d$.  In case $V= \U(1)^{\oplus n}$, 
we write $\Gamma^{d,n} = \Gamma^{d,V}$.
\end{enumerate}
\end{examples}

Given $V \in \smod_{\C(1)}$, notice that we have a canonical 
isomorphism 
\begin{equation}\label{eq:div type II}
\Gamma^{d,1}(V) \simeq \Gamma^d V,
\end{equation}
 since 
 $$\Gamma^d \Hom_{\C(1)}(\U(1),V)\ =\ \Gamma^d\Hom_{\C(1)}(\C(1),V)\ \simeq\ \Gamma^d V.$$

\subsection{Duality}
Suppose $\tau$ is an antiautomorphism of a superalgebra $\B$, and 
let $V \in \smod_\B$. 
Then we can make the dual space $V^\vee$ 
into a right $\B$-supermodule by defining
$$ \la f.a,v \ra = \la f, v.\tau(a)\ra \qquad (a \in \B, f \in V^\vee, v\in V).$$
We denote the resulting supermodule by $V^{\tau, \vee}$.  If 
$V,W \in \smod_{\B}$ and $\varphi \in \Hom_{\C(1)}(V,W)$, then let 
$\varphi^\vee: W^{\tau, \vee} \rightarrow V^{\tau, \vee}$ be defined by
$$\la \varphi^\vee(f), v \ra \ =\ (-1)^{|\varphi| |f|}\la f, \varphi(v) \ra $$
for all $f \in V^\vee$, $v\in V$.
Then $\varphi^\vee \in \Hom_\B(W^{\tau,\vee}, V^{\tau,\vee})$, and we 
furthermore have a natural isomorphism 
\begin{equation}\label{eq:hom dual}
\Hom_\B(V,W)\ \simeq\ \Hom_\B(W^{\tau, \vee}, V^{\tau,\vee}).
\end{equation}

Given any $\svec_\k$-enriched category $\V$, let us write 
$\V^{\op,-} = (\V^-)^\op$ to denote the opposite category of $\V^-$.
Now let $\V= \smod_\B$.
Then (\ref{eq:hom dual}) gives an equivalence of categories
\begin{equation}\label{eq:duality2}
(\hspace{.3cm} )^{\tau, \vee}: (\smod_{\B})^{\op,-} \equi \smod_{\B}.
\end{equation}

An antiautomorphism  $\tau$ of $\B$ induces an antiautomorphism $\tau_2$
of $\B\otimes \B$ by setting $\tau_2(a\otimes b)= (-1)^{|a| |b|} \tau(a)\otimes \tau(b)$.
In general, this gives an antiautomorphism $\tau_d$ of $\B^{\otimes d}$ for all
$d\ge1$.  If $V,W \in \smod_\B$, we have a canonical isomorphism of 
$\B\otimes \B$-supermodules
\begin{equation*}
V^{\tau,\vee}\otimes W^{\tau,\vee} \simeq (V\otimes W)^{\tau_2,\vee}
\end{equation*}
given by 
\begin{equation}\label{eq:tens power}
\la f \otimes g, v\otimes w \ra = (-1)^{|g| |v|} \la f,v \ra \la g, w \ra
\end{equation}
for all $f\in V^{\tau,\vee}$, $g \in W^{\tau,\vee}$, $v\in V$, $w\in W$.

Suppose now that $\B$ is a simple finite dimensional superalgebra.
Let us fix generators $s_i= (i\ i+1)\in \Si_d$ for $i=1,\dots, d-1$.  Then $\tau_d$
extends uniquely to an antiautomorphism of $\B\wr \Si_d$, also denoted $\tau_d$,
such that $\tau_d(s_i) = s_i$ for $i=1, \dots, d-1$.  So that the equivalence 
(\ref{eq:duality2}) with respect to $\B\wr \Si_d$ and $\tau_d$ induces a correspoding
equivalence
$$ (\ )^{\tau,\vee}: (\Gamma^d\V)^{\op,-} \equi \Gamma^d\V$$
for $\V= \smod_{\B}$.

\begin{example}
If $\B= \k$ or $\C(1)$, then $\tau(a)=a$ $(^\forall a\in \B)$ defines an antiautomorphism
of $\B$.  Hence we have equivalences
$$ (\ )^\vee: (\Gamma^d_\M)^{\op,-} \equi \Gamma^d_\M \qquad \mbox{and}
\qquad (\ )^\vee: (\Gamma^d_\Q)^{\op,-} \equi \Gamma^d_\Q.$$
\end{example}

Again let $\V= \smod_\B$ for a finite dimensional simple superalgebra $\B$.  Suppose
$T\in \Fct_\k(\Gamma^d\V, \svec_\k)$.  Then define the {\em Kuhn dual}
$$T^\# \in \Fct_\k\left((\Gamma^d\V)^{\op,-}, (\svec_\k)^{\op,-}\right) =\ \Fct_\k\left(\Gamma^d\V, \svec_\k \right)$$
by setting
$T^\#(V)= T(V^{\tau,\vee})^\vee$ for all $V\in \V$.  If $\B= \k$ or $\C(1)$, 
this gives an equivalence
\begin{equation}\label{eq:duality4}
 (\ )^\#: (\Pol_d^\dagger)^{\op,-} \equi \Pol_d^\dagger,
 \end{equation}
for $\dagger =$ I or II, respectively.

\smallskip

As an example, for $V\in \Gamma^d_\M$ (resp. $\Gamma^d_\Q$), we define 
$S^{d,V}:= (\Gamma^{d,V})^\#$.  In particular, let us write 
$S^{d, m|n} = S^{d, \k^{m|n}}$ and $S^{d,n}= S^{d, \U(1)^{\oplus n}}$.  It then 
follows from equation (\ref{eq:duality1}) that we have canonical isomorphisms
$$ S^{d, 1|0}(M) \simeq S^d M, \quad S^{d, 1}(V) \simeq S^d V$$
for all $M \in \svec_\k$, $V \in \smod_{\C(1)}$ respectively.

\subsection{Yoneda's lemma}

We have the following analogue of Yoneda's lemma in our setting.
\begin{lemma}
Suppose that $T \in \Pol_d^\I$ and $T' \in \Pol_d^\II$.  Then we have natural 
isomorphisms:
\begin{align*}
{\rm (i) }\ &\hom_{\Pol^\I_d}(\Gamma^{d,M}, T) \simeq T(M) \quad (M\in \svec_\k),\ {\rm and}\\
{\rm (ii) }\ &\hom_{\Pol^\II_d}(\Gamma^{d,V}, T') \simeq T'(V) \quad (V\in \smod_{\C(1)}).
\end{align*}
\end{lemma}

It follows that $\Gamma^{d,M}, \Gamma^{d,V}$ are projective objects of 
$(\Pol^\I_d)_\ev, (\Pol^\II_d)_\ev$ respectively.  On the other hand, the dual objects 
$S^{d,M}, S^{d,V}$ are injective by Kuhn duality (\ref{eq:duality4}).

\subsection{Tensor products}

Given nonnegative integers $d$ and $e$, we have an embedding 
$\Si_d \times \Si_e \hookrightarrow \Si_{d+e}$.  This induces an embedding 
\begin{equation}\label{eq:embed}
\Gamma^{d+e}M \hookrightarrow \Gamma^{d}M\otimes \Gamma^{e}M,
\end{equation}
 for any $M\in \svec_\k$, given by the composition of the following maps
\begin{align*}
 \Gamma^{d+e}M = (M^{\otimes(d+e)})^{\Si_{d+e}}\ \subseteq\ &(M^{\otimes d+e)})^{\Si_d\times \Si_e}\\
  \simeq\ &\ (M^{\otimes d})^{\Si_d}\otimes (M^{\otimes e})^{\Si_e} = \Gamma^dM \otimes \Gamma^eM.
\end{align*}
Now we may consider the categories $\Gamma^d_\M \otimes \Gamma^e_\M$, 
$\Gamma^d_\Q\otimes \Gamma^e_\Q$ whose objects are the same as $\svec_\k$, $\smod_{\C(1)}$ 
and whose morphisms are of the form
$$\hom_{\Gamma^d_\M}(M,N)\otimes \hom_{\Gamma^e_\M}(M,N), 
\quad \hom_{\Gamma^d_\Q}(V,W)\otimes \hom_{\Gamma^e_\M}(V,W)$$  
respectively for $M,N \in \svec_\k$ and $V,W\in \smod_{\C(1)}$.  Then, one may show that 
(\ref{eq:embed}) yields embeddings of categories
\begin{equation}\label{eq:embed2}
 \Gamma^{d+e}_\M \hookrightarrow \Gamma^d_\M \otimes \Gamma^e_\M, \qquad  
 \Gamma^{d+e}_\Q \hookrightarrow \Gamma^d_\Q \otimes \Gamma^e_\Q.
\end{equation}

Now suppose $S\in \Pol^{(\dagger)}_d$,  $T\in \Pol^{(\dagger)}_e$ ($\dagger=\mathrm{I,II}$). Let 
$S\boxtimes T \in \Pol^{(\dagger)}_{d+e}$ denote the functor defined by setting: for $V,W \in \svec_\k$ 
(resp. $\smod_{\C(1)}$) and $\varphi:V\rightarrow W$ a $\k$-linear (resp. $\C(1)$-linear) map,
$$(S\boxtimes T)(V) := S(V)\otimes T(V) \mbox{ and } (S\boxtimes T)(\varphi) := S(\varphi)\boxtimes T(\varphi),$$
respectively.
Then (\ref{eq:embed2}) induces bifunctors:
$$-\otimes-: \Pol^{(\dagger)}_d \times \Pol^{(\dagger)}_e \rightarrow \Pol^{(\dagger)}_{d+e} 
\qquad (\dagger = \mathrm{I,II}), $$
which respectively send $S\times T \mapsto S\boxtimes T$.
\medskip

\section{Strict polynomial functors, Schur superalgebras and Sergeeev duality}\label{sec:schursalg}
We show that the categories of strict polynomial functors of types I and II 
defined above  are equivalent to categories of supermodules for the 
Schur superalgebras $\mathcal{S}(m|n,d)$ and $\Qc(n,d)$, 
respectively.  We then describe a functorial analogue of Sergeev duality for 
type II strict polynomial functors.

\subsection{Equivalences of categories}
Let $M \in \svec_{\k}$.  If $v\in M$, we write 
$v^{\otimes d} = v\otimes \cdots \otimes v$ 
($d$ factors).  Suppose that 
$I = (d_1, \dots, d_s)$ is a tuple of positive 
integers, and let $\mathfrak{S}_I$ denote 
the subgroup 
$\mathfrak{S}_{d_1}\times \cdots \times \mathfrak{S}_{d_s} \subseteq \mathfrak{S}_{|I|}$, 
where $|I| = \sum d_i$.  Given distinct 
nonzero elements $v_1, \dots, v_s \in M_0$, 
we define the new element
\[
(v_1, \dots, v_s; I)_0 := 
\sum_{\sigma \in \mathfrak{S}_{|I|}/\mathfrak{S}_{I}} (v_1^{\otimes d_1} \otimes \cdots \otimes v_s^{\otimes d_s}).\sigma,
\]
which belongs to 
$(M_0^{\otimes |I|})^{\mathfrak{S}_{|I|}}$.  Similiarly, 
if $v_1', \dots, v_t' \in M_1$, we define the (possibly 
zero) element
\[
(v_1', \dots, v_t')_1 := \sum_{\sigma \in \mathfrak{S}_{t}} (v_1' \otimes \cdots \otimes v_t').\sigma,
\]
which belongs to $(M_1^{\otimes t})^{\mathfrak{S}_t}$.

\begin{lemma}\label{lem:divided}
Let $M\in \svec_{\k}$, and suppose 
$\{e_1^{(0)}, \dots, e_m^{(0)}\}$, 
$\{e_1^{(1)}, \dots, e_n^\mathrm{(1)}\}$ are ordered 
bases of $M_0, M_1$ respectively.  Then 
$\Gamma^d M$ has a basis given by the set of all 
elements of the form
$$ (e^{(0)}_{i_1}, \dots, e^{(0)}_{i_s}; I)_0 \otimes (e^{(1)}_{j_1}, \dots, e^{(1)}_{j_t})_1,
$$
such that: $|I| +t = d$, $1\le i_1< \cdots< i_s\le m$ and 
$1\le j_1< \cdots< j_t\le n$. 
\end{lemma}

\begin{proof}
It follows from (\ref{eq:div1})  
that we have isomorphisms of superspaces
\begin{equation}\label{eq:Gamma}
\Gamma^d M\ \simeq\ \bigoplus_{k + l = d} D^k(\overline{M_0}) \otimes \Lambda^l(\overline{M_1})
\ \simeq \bigoplus_{k+l =d} \Gamma^k M_0 \otimes \Gamma^l M_1,
\end{equation}
for each $d \ge 0$.  

One may check by comparison with Proposition 4 of
\cite[Ch.IV, \S5]{Bourbaki} that the set 
\[ \{ (e_{i_1}^{(0)}, \dots, e_{i_s}^{(0)}; I)_0 ;\ |I| = k \mbox{ and } 1\le i_1< \cdots< i_s \le m\}
\]  
is a basis of $\Gamma^k M_0$.  It is also 
not difficult to verify that
\[ \{ (e_{j_1}^{(1)}, \dots, e_{j_l}^{(1)})_1;\ 1 \le j_1< \cdots< j_l \le n\}
\]  
is a basis of $\Gamma^l M_1$.  The 
lemma then follows from (\ref{eq:Gamma}).
\end{proof}

We are now ready to prove the main theorem.

\begin{theorem}\label{thm:main1}
Assume $m,n\ge d$.   Then evaluation on $\k^{m|n}$, 
$\U(1)^{\oplus n}$ yields equivalences of categories: 
$$\Pol^\I_d\ \equi\ {}_{\mathcal{S}(m|n,d)}\smod, \quad  \Pol^\II_d\ \equi\ {}_{\Qc(n,d)}\smod$$
respectively.
\end{theorem}

 \begin{proof}
We prove only the second equivalence, since the proof of the first equivalence is similar.
Recall that  
$$\Qc(n,d)=\ \eEnd_{\Gamma^d_\Q}( \U(1)^{\oplus n}).$$
According to Proposition \ref{prop:appendix}, it suffices to show that the map 
 induced by composition,
 \begin{equation}\label{eq:comp}
 \hom_{\Gamma^d_\Q}(\U(1)^{\oplus n}, W) \otimes \hom_{\Gamma^d_\Q}(V, \U(1)^{\oplus n}) 
 \rightarrow \hom_{\Gamma^d_\Q}(V,W)
 \end{equation}
is surjective for all $V,W \in \Gamma^d_\Q$.  From Example \ref{ex:cliff2} in Section 2, 
it follows that for any $r\in \Z_2$ there exist bases $(x^{(r)}(j,i))$, $(y^{(r)}(k,j))$ and 
$(z^{(r)}(k,i))$ of $\Hom_{\C(1)}(V, \U(1)^{\oplus n})_r$, 
$\Hom_{\C(1)}(\U(1)^{\oplus n}, W)_r$ 
and $\Hom_{\C(1)}(V, W)_r$ respectively, such that:
\[y^{(r)}(k,j_1) \circ x^{(r')}(j_2,i)\  =\ \delta_{j_1,j_2}\ z^{(r+r')}(k,i),\]
for $r,r' \in \Z_2$, where $\delta_{j_1,j_2}$ is the Kronecker delta.

To prove surjectivity, it suffices to show for $1\le s,t \le d$ that each element of the form:
\begin{equation}\label{eq:element}
 (z^{(0)}(k_1,i_1), \dots, z^{(0)}(k_s,i_s); I)_0 \otimes (z^{(1)}(k_1',i_1'), \dots, z^{(1)}(k_t',i_t'))_1
\end{equation}
in $\Gamma^d \Hom_{\C(1)}(V,W) \simeq \hom_{\Gamma^d_\Q}(V,W)$ lies
 in the image of (\ref{eq:comp}), since $\Gamma^d \Hom_{\C(1)}(V,W)$ is 
 spanned by such elements according Lemma \ref{lem:divided}.
Now since $n\ge d$, we have $n\ge s$ and $n\ge t$.  Thus we may choose 
{\em distinct} indices $j_1, \dots, j_s$ (resp. $j_1', \dots, j_t'$) to form the element
\begin{align*}
 &(y^{(0)}(k_1,j_1), \dots, y^{(0)}(k_s,j_s); I)_0 \otimes (y^{(1)}(k_1',j_1'), \dots, y^{(1)}(k_t',j_t'))_1\\
 &\otimes\ (x^{(0)}(j_1,i_1), \dots, x^{(0)}(j_s,i_s); I)_0 \otimes (x^{(1)}(j_1',i_1'), \dots, x^{(1)}(j_t',i_t'))_1,
\end{align*}
which is sent to the element (\ref{eq:element}) under the map induced by composition 
in $\Gamma^d_\Q$.
\end{proof}

From the previous thereom and the classifications given in \cite{Donkin}, 
\cite{BKprojective} we obtain the following corollary.  By a {\em partition}
we mean an infinite nonincreasing sequence 
$\lambda=( \lambda_1, \lambda_2, \dots)$ of nonnegative integers such 
that the sum $|\lambda|= \sum \lambda_i$ is finite.  Let $\mathbb{P}$ denote the 
set of all partitions.

\begin{corollary}
The set of distinct isomorphism classes of simple objects of $\Pol^\I_d$ 
is in bijective correspondence with the set of pairs
$$\{ (\lambda, \mu)\ ;\ \lambda,\mu\in \mathbb{P} \mbox{ and } |\lambda|+ p|\mu|= d\}.$$

Now suppose in addition that the field $\k$ is algebraically closed.  Then 
the set of classes of simple objects of $\Pol^\II_d$ is in bijective correspondence 
with the set of partitions 
$$\{ \lambda\in \mathbb{P}\ ;\ |\lambda|=d, \mbox{ and } \lambda_i \neq \lambda_{i+1} \mbox{ if } p \nmid \lambda_i\}.$$ 
\end{corollary}
\smallskip

\subsection{Spin polynomial functors and Sergeev duality}
In this section we limit our attention to the objects $T\in \Pol^\II_d$.  
We may refer to such strict polynomial functors as  {\em spin 
polynomial functors}.  The explanation for this term is given by 
Theorem \ref{thm:main2} below, which describes a relationship 
between $\Pol^\II_d$ and finite dimensional representations 
of the Sergeev superalgebra, which is ``super equivalent" to the 
spin symmetric group algebra $\k^-\Si_d$ (cf. \cite{BKprojective}).

Let us denote $$\Pol^\II = \bigoplus_{d\ge 0} \Pol^\II_d.$$
There is a bifunctor $\ \Pol^\II \times \Pol^\II \rightarrow \Pol^\II$ 
given by the (external) tensor product 
$$ -\otimes- : \Pol^\II_d \times \Pol^\II_e \rightarrow \Pol^\II_{d+e},$$
defined in Section 2.5.

Suppose $M,N \in \svec_\k$.  Then $\Gamma^*(\ )$ satisfies the {\em exponential property}
\begin{equation}\label{eq:div4}
\Gamma^*(M\oplus N)\ \cong\ \Gamma^*M \otimes \Gamma^*N,
\end{equation}
which follows from (\ref{eq:div1}) and the corresponding properties for $D^*(\ )$ and 
$\Lambda^*(\ )$.
It follows from (\ref{eq:Gamma}) and (\ref{eq:div4}) that 
\begin{equation}\label{eq:exp}
\Gamma^d(M\oplus N) \ = \ \bigoplus_{i=0}^d \Gamma^{d-i}M \otimes \Gamma^iN.
\end{equation}
Recall the objects $\Gamma^{d,n} \in \Pol^\II_d$ which are projective by 
Yoneda's lemma (see Section \ref{sec:strictpoly}).  It  follows from 
(\ref{eq:exp}) that we have a decomposition 
\begin{equation}\label{eq:decomp.}
\Gamma^{d,m+n} \simeq \bigoplus_{i+j=d} \Gamma^{i, m} \otimes \Gamma^{j, n}
\end{equation}
of strict polynomial functors.

Now let $\Lambda(n,d)$ denote the set of all tuples 
$\lambda=(\lambda_1, \dots, \lambda_n)\in (\Z_{\geq 0})^n$ 
such that $\sum \lambda_i = d$.  Given $\lambda \in \Lambda(n,d)$, 
we will write 
$\Gamma^\lambda = \Gamma^{\lambda_1,1} \otimes \cdots \otimes \Gamma^{\lambda_n,1}$.
By (\ref{eq:decomp.}) and induction, we have a canonical isomorphism
\begin{equation}\label{eq:lambda}
\Gamma^{d,n} \simeq \bigoplus_{\lambda\in \Lambda(n,d)} \Gamma^\lambda.
\end{equation}
It follows that the objects $\Gamma^\lambda$ are projective in $\Pol^\II_d$.

Let $\omega=(1,\dots, 1) \in \Lambda(d,d)$.  Then $\Gamma^\omega = \otimes^d$, 
and $\otimes^d$ is  a projective object of $\Pol^\II_d$.  We have the following 
analogue of \cite[Theorem 6.2]{BKprojective}.

\begin{theorem}\label{thm:main2}
Assume $n \ge d$.
\begin{itemize}
\item[(i)] The left $\Qc(n,d)$-supermodule 
$V^{\otimes d} \simeq \hom_{\Pol^\II_d}(\Gamma^{d,n}, \otimes^d)$ 
is a projective object of $_{\Qc(n,d)}\smod$.
\smallskip
\item[(ii)] There is a canonical isomorphism of superalgebras: 
$$\eEnd_{\Pol^\II_d}(\otimes^d) \cong \W(d).$$

\item[(iii)] We have an exact functor 
$$\hom_{\Pol^\II_d}(\otimes^d, -): \Pol^\II_d \rightarrow {}_{\W(d)}\smod.$$
\end{itemize}
\end{theorem}

\begin{proof}
(i) follows from Theorem \ref{thm:main1} and the fact that
 $\bigotimes^d$ is a projective object of $\Pol^{\II}_d$, 
 and  (ii) follows from Theorem \ref{thm:main1} and 
 \cite[Theorem 6.2.(iii)] {BKprojective}.  Finally, (iii) is a 
 direct consequence of (i) and (ii).
\end{proof}

\begin{remark}
One may refer to the functor in Theorem \ref{thm:main2}.(iii) 
as the {\em Sergeev duality functor}.  A similar functor 
related to classical Schur-Weyl duality was studied in 
\cite{HY} in the context of $\mathfrak{g}$-categorification.
\end{remark}
\medskip

\section{Categories of $\mathfrak{ssch}_\k$-enriched functors}\label{sec:enrich}
In this section, we provide an alternate definition of strict polynomial functors 
which is a `super analogue' of Friedlander and Suslin's original definition 
\cite[Definition 2.1]{FS}.  We also introduce categories $\ePol^\I_d$ and 
$\ePol^\II_d$ whose objects are homogeneous {\em $\ssch_\k$-enriched functors} 
between a pair of $\ssch_\k$-enriched categories.
Familiarity with the notation and material from Appendix B will be assumed 
throughout this section.

\subsection{Definition of $\ssch_\k$-enriched functors}
Recall that we may identify $\ssch_\k$ as a full subcategory of the functor 
category $\Fct(\salg_\k, \mathfrak{sets})$.  Given superschemes $X,Y\in \ssch_\k$, 
the functor $X\times Y$ is again a superscheme.  Let $I_0$ be a constant functor 
such that $I_0(\A)=\{0\}$ for all $\A \in \salg_\k$.  Then $I_0$ is an affine 
superscheme with $\k[I_0]=\k$.  The monoidal structure on the category 
$\mathfrak{sets}$ with respect to direct product induces a corresponding 
(symmetric) monoidal structure on  $\ssch_\k$, such that $I_0$ is an identity 
element.

Let $X,Y\in \ssch_\k$, with $X$ an affine superscheme.  An analogue 
of  \cite[I.1.3]{Jantzen} (Yoneda's lemma for ordinary schemes) gives a 
bijection
\begin{equation}\label{eq:Yoneda}
{\rm hom}_{\ssch_\k}(X,Y)\ \equi\ Y(\k[X]).
\end{equation}
Let $\B$ an associative superalgebra, and suppose $U,V,W \in {}_\B\smod$.
Then, there is a natural transformation
$$\Hom_\B(V,W)_a \times \Hom_\B(U,V)_a \rightarrow \Hom_\B(U,W)_a$$
given by the isomorphism (\ref{eq:tensor1}) and composition of $\A$-linear 
maps, for all $\A\in \salg_\k$.  We also have for each $V\in \svec_\B$ a 
natural transformation
$$ j_V: I_0 \rightarrow \End_\B(V)_a$$
which is the element of $\hom_{\ssch_\k}(I_0, \End_\B(V)_a)$ mapped onto 
$\Id_V \in \End_\B(V)_0$ under the bijection (\ref{eq:Yoneda}).
It then may be checked that we obtain an $\ssch_\k$-enriched category 
$_\B\esmod$ (in the sense of \cite{Kelly}) which has the same objects as 
$_\B\smod$ and {\em hom-objects}
$$ \hom_{_\B\esmod}(V,W) = \Hom_\B(V,W)_a$$  
for all $V,W\in {}_\B\smod$.  If $\B= \k$, we write $_\B\esmod= \esvec_\k$.

\begin{definition}
Suppose $\B$ is an associative superalgebra.  Let $\V= {}_{\B}\smod$, and let $\eV$
 denote the corresponding $\ssch_\k$-enriched category.  A 
 {\em $\ssch_\k$-enriched functor} (or {\em $\ssch_\k$-functor}) 
 $$T: \eV \rightarrow \esvec_\k$$ 
 consists of an assignment
$$ T(V) \in \svec_\k \quad (^\forall V\in \V)$$
and a morphism of superschemes
$$T_{V,W}: \Hom_\B(V,W)_a \rightarrow \Hom_\k(T(V),T(W))_a \quad (^\forall V,W\in \V),$$
such that the following two diagrams commute for all $U,V,W \in \V$:

$$\xymatrix{
I_0\ar[rd]_-{j_{T(V)}}\ar[r]^-{j_V}& \End_\B(V)_a\ar[d]^-{T_{V,V}}\\
&\End_\k(T(V))_a,
}$$
and
$$\xymatrix{
\Hom_\B(V,W)_a\times \Hom_\B(U,V)_a\ar[d]_-{T_{V,W}\times T_{U,V}} \ar[r] &\Hom_\B(U,W)_a\ar[d]^-{T_{U,W}} \\
\Hom_\k(T(V),T(W))_a \times \Hom_\k(T(U),T(V))_a \ar[r] &\Hom_\k(T(U),T(W))_a,
}$$
with horizontal maps being given by composition in $\eV$ and $\esvec_\k$, respectively.
\end{definition}

\subsection{The categories $\ePol_d^{(\dagger)}$ for $\dagger =$ I, II}
Notice that if $f: M\rightarrow N$ is an even linear map of vector 
superspaces, then $f$ may be identified with the associated 
natural transformation $\eta_f:M_a \rightarrow N_a$ which is 
given by the $\k$-linear maps
$$ \eta_f(\A) = f\boxtimes 1_\A: (M\otimes \A)_0 \rightarrow (N\otimes \A)_0,$$
for all $\A \in \salg_\k$.

\begin{definition}
Let $\V = {}_\B\smod$, and let $\eV= {}_\B\esmod$.
Suppose that $S,T: \eV\rightarrow \esvec_\k$ are both 
$\ssch_\k$-functors.  Then a {\em $\ssch_\k$-natural 
transformation}, $\alpha: S\rightarrow T$, is defined to 
be a collection of even $\k$-linear maps
$\alpha_V: S(V)\rightarrow T(V)$
such that the following diagram commutes for all $V,W \in \eV$:
$$\xymatrix{
\Hom_\B(V,W)_a \ar[d]_-{T_{V,W}} \ar[r]^-{S_{V,W}} &\Hom_\k(S(V),S(W))_a\ar[d]^-{\alpha_W\circ-} \\
\Hom_\k(T(V),T(W))_a  \ar[r]_-{-\circ\alpha_V} &\Hom_\k(S(V),T(W))_a,
}$$
where we have identified the even linear maps, $\alpha_W\circ -$ and 
$-\circ \alpha_V$, with their corresponding natural transformations as 
described in the preceding paragraph.  Denote by 
$\Fct_{\ssch_\k}(\eV, \esvec_\k)$ the category of all $\ssch_\k$-functors, 
$T: \eV \rightarrow \esvec_\k$, and $\ssch_\k$-natural transformations.
\end{definition}

Let $\eV= {}_\B\esmod$, and suppose $V\in {}_\B\smod$.  Given 
$T\in \Fct_{\ssch_\k}(\eV, \esvec_\k)$
consider the algebraic supergroup $G=GL_{\B,V}$ and recall that
$End_{B,V}= \End_\B(V)_a$.  Then, by the definition of $\ssch_\k$-functor, 
the induced natural transformation
$T_{V,V}: End_{\B,V} \rightarrow End_{\k, T(V)}$
 restricts to a natural transformation of supergroups,
$$\eta^{T,V}:G \rightarrow GL_{\k,T(V)},$$
which preserves identity and products.  Hence $\eta^{T,V}$ is 
a representation of the supergroup $G$.  

Now $T(V)$ may also be considered as a $G$-supermodule 
with a corresponding structure map $$\Delta_{T,V}: T(V) 
\rightarrow T(V)\otimes \k[G].$$  Notice that for any 
$M,N \in \svec_\k$, Yoneda's lemma gives a canonical isomorphism
\begin{equation}\label{eq:affine}
 \hom_{\ssch_\k}(M_a,N_a)\ \simeq\ (N\otimes \k[M_a])_0
 \end{equation}
for the corresponding affine superschemes.  Using 
(\ref{eq:affine}), let us identify the natural transformation 
$T_{V,V}$ with an element of the set 
$$(\End_\k(T(V)) \otimes \k[ \eEnd_\V(V)_a])_0.$$
It is then not difficult to see how $T_{V,V}$ gives rise to the 
structure map $\Delta_{T,V}$.  Hence the image of $\Delta_{T,V}$ 
lies in $T(V)\otimes\k[End_{\B,V}]$, and $T(V)$ is a polynomial representation 
of $G$.

\begin{definition}\label{def:enriched}
Let $\eV = {}_\B\esmod$.  We define $\Fct_{\ssch_\k}(\eV, \esvec_\k)_{(d)}$ 
to be the full subcategory of $\Fct_{\ssch_\k}(\eV, \esvec_\k)$ consisting 
of all $\ssch_\k$-enriched functors $T: \eV \rightarrow \esvec_\k$ such that 
$$T_{V,W} \in (\Hom_\k(T(V),T(W)) \otimes S^d (\Hom_\B(V,W)^\vee)_0$$
for all $V,W \in {}_\B\smod$ (where we have identified both sides of (\ref{eq:affine})).
We write
$$\ePol^{\I}_d= \Fct_{\ssch_\k}(\esvec_\k, \esvec_\k)_{(d)},  \mbox{ and }$$ 
$$\ePol^{\II}_d= \Fct_{\ssch_\k}(_{\C(1)}\esmod, \esvec_\k)_{(d)}.$$
From Theorem \ref{thm:main3} below, it follows that these categories are
equivalent to $\Pol^\I_d$ and $\Pol^\II_d$ respectively.
\end{definition}
\smallskip

\subsection{Polynomial representations of $GL(m|n)$ and $Q(n)$}
Suppose $m,n$ are fixed nonnegative integers.    Let us write 
$\Sc^\mathrm{I} = \Sc(m|n,d)$ and $\Sc^\mathrm{II} = \Qc(n,d)$.
We also write $G^\mathrm{I} = GL(m|n)$ and $G^\mathrm{II} = Q(n)$.
If $\dagger =$ I, let 
$V_l= V_r = \k^{m|n}\in \svec_\k$,and if $\dagger=$ II, let 
$V_l= \U_l(1)^{\oplus n}\in _{\C(1)}\smod$ and 
$V_r= \U(1)^{\oplus n}\in \smod_{\C(1)}$. 

\begin{theorem}\label{thm:main3}
 Suppose $m,n \ge d$.  Then we have equivalences of categories:
$$  \mathrm{(i)}\  \Psi:  \mathfrak{pol}_d(G^\dagger) \equi {}_{\Sc^\dagger}\smod, 
\qquad \mathrm{(ii)}\   \Phi: \ePol^{(\dagger)}_d \equi \Pol^{(\dagger)}_d,$$
for $\dagger =$ I, II respectively.  
\end{theorem}

\begin{proof}
Proof of (i).  Let $\B=\k,\ \C(1)$ if $\dagger =$ I, II respectively. 
It suffices to show that we have an isomorphism of superalgebras
$$ \Sc^\dagger \cong (\k[\End_\B(V_l)]_d)^\vee.$$
Using  Proposition \ref{prop:cosalg}.(iii), 
(\ref{eq:minus1}), (\ref{eq:k minus}) and (\ref{eq:minus2}), we have
\begin{align*}
(\k[End_{\B,V_l}]_d)^\vee &=\ S^d(\End_\B(V_l)^\vee)^\vee\\
&\cong\ \Gamma^d(\End_\B(V_l)^-)\ \cong\ \Gamma^d(\End_\B(V_r))\quad =\ \Sc^\dagger.
\end{align*}

Proof of (ii).  
Let $\V= {}_\B\smod$ for $\B$ as above.  Then we identify $\V^-$
with $\svec_\k,\ \smod_{\C(1)}$ respectively, using (\ref{eq:k minus})
and (\ref{eq:minus2}).  Hence the objects of $\V$ are identical to the 
objects of either $\V^-$ or $\Gamma^d(\V^-)$ respectively. 

Suppose $T\in \ePol^\dagger_d$.  We will define a functor
$\Phi(T): \Gamma^d(\V^-) \rightarrow \svec_\k$.  Given $V\in \V$, let 
$\Phi(T)(V)= T(V)\in \svec_\k$. 
Now suppose $V,W\in \V^-$.  We have a map
\begin{align*}
T_{V,W}& \in\ S^d(\hom_\V(V,W)^\vee)\otimes \Hom\left(T(V),T(W)\right)\\
&\cong\ \Hom\left(\Hom(T(V),T(W))^\vee,\ S^d(\hom_\V(V,W)^\vee)\right)\\
&\cong\ \Hom\left(\Gamma^d \hom_{\V^-}(V,W),\ \Hom(T(V),T(W))\right).
\end{align*}
Let $\Phi(T)_{V,W}: \hom_{\Gamma^d(\V^-)}(V,W) \rightarrow \Hom(T(V),T(W))$
denote the image of $T_{V,W}$ under the above isomorphism.  Then it may
be checked that $\Phi(T) \in \Fct(\Gamma^d(\V^-), \svec_\k)$, and that this 
gives an equivalence of categories
$$\Phi: \ePol^\dagger_d \equi \Pol^\dagger_d$$
which maps $T\mapsto \Phi(T)$.
\end{proof}

\begin{corollary} Suppose $m,n \ge d$, and let $V_l, V_r$ be
as above.  Then we have a commutative diagram 
$$\xymatrix{
\ePol^{(\dagger)}_d\ar[d] \ar[r]^-{\Phi}  &\Pol^{(\dagger)}_d \ar[d]\\
\mathfrak{pol}_d(G^\dagger)  \ar[r]^-{\Psi}   &{}_{\Sc^\dagger}\smod 
 &(\dagger= \mathrm{I,II}),}
$$
where the vertical arrow on the left is evaluation at $V_l$ and
the vertical arrow on the right is evaluation at $V_r$.  In particular,
evaluation at $V_l$ gives an equivalence $\ePol^{(\dagger)}_d
\equi \pol_d(G^\dagger)$ for $\dagger =$ I, II respectively.
\end{corollary}
\begin{proof}
We know that the vertical arrow on the left is an equivalence by
Theorem \ref{thm:main1}.  It is then not difficult to see from the 
definitions of the functors $\Phi$ and $\Psi$ that the diagram
is combative.  Hence, from Theorem \ref{thm:main3} the 
commutativity implies that the evaluation at $V_r$ also gives
an equivalence.
\end{proof}
\medskip

\appendix

\section{Representations of $\svec_\k$-enriched categories}\label{app:repns}

Recall that $\k$ is a field of characteristic not equal 2, and $\svec_{\k}$ 
denotes the category of finite dimensional vector superspaces over $\k$.  
Suppose $\V$ is a category enriched over $\svec_{\k}$.  
In this appendix we describe the relationship between the following two categories:
\begin{itemize}
\item[(i)] The category $\V \text{-}\smod = \Fct_\k(\V, \svec_\k)$ of all $\k$-linear 
{\em representations} of $\V$.  It consists of all even $\k$-linear functors $\V\to\svec_\k$.
\medskip
\item[(ii)] If $P\in\V$, then $\E= \End_\V(P)$ is an associative superalgebra 
with product given by composition.  We may then consider the category 
$_\E\smod$ of finite dimensional left supermodules over $\E$.
\end{itemize}

The categories $\V \text{-}\smod$ and $_\E\smod$ are both $\svec_\k$-enriched 
categories.  We denote by $(\V\text{-}\smod)_{\rm ev}$, $(_\E\smod)_{\rm ev}$ 
the corresponding even subcategories.  Recall from Section \ref{sec:salg} that 
$(_\mathcal{A}\smod)_{\rm ev}$ is an abelian category for any finite dimensional 
superalgebra $\mathcal{A}$.  In particular, $(_\E\smod)_{\rm ev}$ and 
$(\svec_\k)_{\rm ev}$ are both abelian categories.  Now since direct sums, products, 
kernels and cokernels can be computed objectwise in (the even subcategory of) the target 
category $\svec_\k$, we see that $(\V\text{-}\smod)_{\rm ev}$ is also an abelian category.

The relationship between $\V\text{-}\smod$ and $_\E\smod$ is given by evaluation 
on $P$.  If $F \in \V\text{-}\smod$, the (even) functoriality of $F$ makes the $\k$-superspace 
$F(P)$ into a supermodule over $\E = \eEnd_\V(P)$.  We thus have an 
evaluation functor:
$$ \V\text{-}\smod \to\ _\E\smod$$
$$ F \mapsto F(P)$$
There is another interpretation of this evaluation functor.  Since the covariant hom-functor 
$h^P:= \hom_\V(P,-)$ is an even $\k$-linear functor, it must belong to $\V\mbox{-}\smod$.  
In this situation, Yoneda's lemma takes the form of an even isomorphism
$$ \hom_{\V\mbox{-}\smod}(h^P, F)\ \simeq\ F(P),$$
for any $F \in \V\mbox{-}\smod$.  In particular, 
$$ \E\ =\ h^P(P)\ \simeq\ \eEnd_{\V\mbox{-}\smod}(h^P).$$
Hence, Yoneda's lemma allows us to interpret ``evaluation at $P$" 
as the functor $\hom_{\V\mbox{-}\smod}(h^P, -): \V\mbox{-}\smod \to{} _\E\smod$.

We are interested to know if there is some condition on $P$ which ensures that evaluation 
is in fact an equivalence of categories.  The next proposition, which is a super analogue of 
\cite[Prop. 7.1]{TouzeRingel}, provides such a criterion.  

Note that the parity change functor, $\Pi: \svec_\k \rightarrow \svec_\k$, induces
by composition a functor $\Pi\circ- : \V\mbox{-}\smod\to \V\mbox{-}\smod$.

\begin{proposition}\label{prop:appendix}
Let $\V$ be an $\svec_\k$-enriched category. Assume that there exists an object $P\in\V$ 
such that for all $X,Y\in\V$, the composition induces a surjective map
$$ \hom_\V(P,Y)\otimes \hom_\V(X,P)\twoheadrightarrow \hom_\V(X,Y)\;.$$
Then the following hold.
\begin{enumerate}
\item[(i)] For all $F\in\V\mbox{-}\smod$ and all $Y\in\V$, the canonical map 
$F(P)\otimes \hom_\V(P,Y)\to F(Y)$ is surjective.
\item[(ii)]  The set  $\{h^P, \Pi h^P\}$ is a projective generator of $(\V\mbox{-}\smod)_{\rm ev}$, 
where  $h^P = \hom_\V(P, -)$ as above.
\item[(iii)] Let $\E = \eEnd_\V(P)$.  Then evaluation on $P$ induces an 
equivalence of categories $\V\mbox{-}\smod\simeq{}  _\E\smod$.
\end{enumerate}
\end{proposition}
\begin{proof}
Proof of (i).  The canonical map is: $f\otimes x \mapsto F(f)(x)$.  By the surjectivity of 
$\hom_\V(P,Y) \otimes \hom_\V(Y,P) \to \hom_\V(Y,Y)$, we may find a finite family 
of maps, $\alpha_i \in \hom_\V(P,Y)$ and $\beta_i \in \hom_\V(Y,P)$, such that
$$ \sum_i \beta_i \circ \alpha_i = {\rm Id}_Y.$$
Now suppose that $y \in F(Y)$.  Then one may check that the element
$$ \sum_i \beta_i \otimes F(\alpha_i)(y) \in \hom_\V(Y,P) \otimes F(P)$$
is sent onto $y$ by the canonical map.

Proof of (ii).  The Yoneda isomorphism $\hom_{\V\text{-}\smod}(h^P, F) \simeq F(P)$ 
ensures that $h^P$ is projective.  One may check that $\Pi h^P$ is then also a 
projective object of $(\V\text{-}\smod)_{\rm ev}$.  Next, by the naturality of the 
canonical map, (i) yields an epimorphism $h^P \otimes F(P) \twoheadrightarrow F$.  
Now $F(P)$ is a finite dimensional superspace.  By choosing a ($\Z_2$-homogeneous) 
basis of $F(P)$, we have $F(P) \simeq \k^{m|n}$ where $\sdim(P) = (m,n)$.  
Hence, there exists an epimorphism 
$\varphi: (h^P)^{\oplus m} \oplus (\Pi h^P)^{\oplus n} \twoheadrightarrow F$,   and we 
may write $\varphi= \varphi_1+ \cdots \varphi_m+ \varphi'_1+\cdots+ \varphi'_n$ for 
some $\varphi_i: h^P\to F$ (resp. $\varphi'_j: \Pi h^P \to F$), where $i=1, \dots, m$ 
(resp. $j=1, \dots, n$).  Then we may finally decompose 
\[F\ =\ \bigoplus_{i=1}^m F_i\ \oplus\ \bigoplus_{j=1}^n F'_j,\]
 where $F_i = {\rm Im}(f_i)$ (resp. $F'_j = {\rm Im}(f'_j)$).  It then follows that 
 $\{h^P, \Pi h^P\}$ is a generating set.

Proof of (iii).  
We first verify that evaluation is fully faithful.  For this purpose, it suffices to check for 
any $F,G \in \V\mbox{-}\smod$ that we have an isomorphism:
$\hom_{\V\mbox{-}\smod}(G,F)\ \simeq\ \Hom_\E(G(P),F(P)).$
Notice that there is a commutative triangle:
$$\xymatrix{
\hom_{\V\mbox{-}\smod}(h^P,F)\ar[d]\ar[r]^-{\simeq}& F(P)\\
\Hom_{\E}(\E,F(P))\ar[ru]_-{\simeq}
},$$
where the horizontal arrow is the Yoneda isomorphism, and the diagonal arrow is 
the isomorphism (\ref{eq:superalg}) from Section \ref{sec:salg}.  Hence the diagram 
induces an (even) isomorphism.  By additivity of homs, we also have an isomorphism
\[ \hom_{\V\mbox{-}\smod}(h^p\otimes\k^{m|n}, F)\ \simeq\ \Hom_\E(\E\otimes\k^{m|n}, F(P)),\] 
for any $m,n \in \mathbb{N}$.  Now by (ii) we may find (for any 
$G \in \V\mbox{-}\smod$) an exact sequence
\begin{equation}\label{eq:exact}
 h^P \otimes \k^{m_2|n_2} \to\ h^P \otimes \k^{m_1|n_1} \to\ G \to\ 0.
 \end{equation}
It then follows by the left exactness of $\hom_{\V\mbox{-}\smod}(-,F)$ and 
$\Hom_\E(-, F(P))$ that evaluation on $P$ is fully faithful.

Next, we verify that evaluation is essentially surjective.  Suppose 
$M\in{} _\E\smod$.  If follows from (\ref{eq:exact}) that 
one may find a presentation of the form
\[
\E\otimes \k^{m_2|n_2}\quad \equi\quad \E\otimes \k^{m_1|n_1}\ \twoheadrightarrow\ M.
\]
Since evaluation on $P$ is fully faithful, there exists a natural transformation 
$\varphi: h^P \otimes \k^{m_2|n_2} \to h^P \otimes \k^{m_1|n_1}$ which 
coincides with $\psi$ upon evaluation at $P$.  Let us define a functor 
$F_M: \V \to \svec_\k$ by $F_M(X) = {\rm coker}(\varphi_X)$.  Then 
$F_M \in \V\mbox{-}\smod$ is a functor whose evaluation at $P$ is 
isomorphic to $M$.  Thus, evaluation at $P$ is essentially surjective.
\end{proof}
\medskip

\section{Superschemes and supergroups}\label{app:superscheme}
We briefly recall the definitions and some basic properties 
of cosuperalgebras, superschemes and supergroups.  For 
more details, see \cite{BKprojective}, \cite{BKmodular} and the references 
therein.

\subsection{Cosuperalgebras}
A {\em cosuperalgebra} is a superspace $\A$ which is a 
coalgebra in the usual sense such that the comultiplication 
$\Delta_\A: \A \rightarrow \A\otimes \A$ and the counit 
$\epsilon:\A\rightarrow \k$ are even linear maps.  The
notions of {\em bisuperalgbra} and {\em Hopf cosuperalgebra}
can be defined similarly.

If $\A$ is a cosuperalgebra, a right $\A$-cosupermodule is a
vector superspace $M$ together with a {\em structure map}
$\Delta_M: M \rightarrow M\otimes \A$ which is an even
linear map that makes $M$ into an ordinary comodule.  
Denote by $\cosmod_\A$ the category of all right 
$\A$-cosupermodules and $\A$-cosupermodule 
homomorphisms (which are just ordinary $\A$-comodule
homormorphisms).

If $\B$ is a finite dimensional associative superalgebra, 
then multiplication in $\B$ gives an even linear map 
$m : \B \otimes \B \rightarrow \B$.  Taking the dual of 
this map we obtain a linear map 
$\Delta = m^\vee: \B^\vee \rightarrow (\B \otimes \B)^\vee = \B^\vee \otimes \B^\vee, $ 
such that
$$ \la \Delta(f), a\otimes b\ra = (-1)^{|\Delta| |f|}\la f, ab\ra = \la f, ab\ra $$
(since $| \Delta| = |m| = 0$), for $a,b\in \B, f\in \B^\vee$.  This map $\Delta$ makes $\B^\vee$ into a 
cosuperalgebra.  

Conversely, suppose that $\A$ is a 
finite dimensional cosuperalgebra.  Then we make 
$\A^\vee$ into a superalgebra by defining the product
$fg$ of $\Z_2$-homogeneous $f,g \in \A^\vee$
as 
$$\la fg,a\ra := \la f\boxtimes g, \Delta_\A(a)\ra,$$
for all $a\in \A$.  Recall from \cite{BKprojective} that 
there is an equivalence (in fact isomorphism) of categories 
between $\cosmod_\A$ and $_{\A^\vee}\smod$.

Suppose $\B$ is an associative superalgebra.  Then
$\Si_d$ acts (on the right) on $\B^{\otimes d}$ via 
superalgebra automorphisms.  Hence, $\Gamma^d \B = 
(\B^{\otimes d})^{\Si_d}$ is also a superalgebra.

Now let $\A$ be a cosuperalgebra.  Since $T^*\A$ is the free
associative superalgebra generated by $\A$ (considered as
a superspace), there is a unique superalgebra homomorphism
$$\Delta: T^*\A \rightarrow T^*\A \otimes T^*\A$$
such that $\Delta(a)= \Delta_\A(a)$ for all $a\in \A$,
and $T^*\A$ is a cosuperalgebra with respect to this
homomorphism.
Similarly, since $S^*\A$ is a free commutative superalgebra, there 
exists a unique superalgebra homomorphism
$$\widetilde{\Delta}: S^*\A \rightarrow S^*\A \otimes S^*\A$$
such that $\widetilde{\Delta}(a)= \Delta_\A(a)$ for all $a\in \A$.
(We note that a tensor product of commutative superalgebras is
commutative.)  The homomorphism $\widetilde{\Delta}$ makes
$S^*\A$ into a cosuperalgebra.

One may check that we have
$$ \Delta(T^d\A) \subseteq T^d\A \otimes T^d\A\quad \mbox{and}\quad
\widetilde{\Delta}(S^d\A) \subseteq S^d\A \otimes S^d\A.$$
Hence, both $T^d\A$ and $S^d\A$ may be considered as
cosuperalgebras by restricting $\Delta$ and 
$\widetilde{\Delta}$ respectively.
\smallskip

\begin{proposition}\label{prop:cosalg}
Suppose $\B$ (resp. $\A$) is a finite dimensional associative 
superalgebra (resp. cosuperalgebra).  
Then we have the following isomorphisms of superalgebras.
\begin{itemize}
\item[(i)] $(\B^\vee)^\vee\ \cong\ \B^-$, where $\B^-$ is defined in Section \ref{sec:salg}.
\medskip
\item[(ii)] $(\A^\vee)^{\otimes d} \cong (\A^{\otimes d})^\vee$
\medskip
\item[(iii)] $S^d(\B^\vee)^\vee \cong \Gamma^d(\B^-)$
\end{itemize}
\end{proposition}

\begin{proof}
For (i) and (ii), the isomorphisms are given by the canonical even
linear isomorphisms (\ref{eq:dual}) and (\ref{eq:tens power}), respectively.
It is then straightforward to check from the definitions that they are 
indeed superalgebra isomorphisms.  For (iii), one may check from parts 
(i) and (ii) that we have the following superalgebras isomorphisms:
\begin{align*}
S^d(\A)^\vee\ =\ ((\A^{\otimes d})_{\Si_d})^\vee &\cong\ ((\A^{\otimes d})^\vee)^{\Si_d}\\ 
&\cong\ ((\A^\vee)^{\otimes d})^{\Si_d}\ =\ \Gamma^d(\B^-).
\end{align*}
\end{proof}

\subsection{Superschemes}
Let $\salg_\k$ denote the category of all commutative superalgebras and even 
homomorphisms.  Also, let $\ssch_\k$ be the category of superschemes as in 
\cite{BKmodular}.  We may identify $\ssch_\k$ with a full subcategory of the 
category $\Fct(\salg_\k, \mathfrak{sets})$ consisting of all functors from 
$\salg_\k$ to $\mathfrak{sets}$.  An {\em affine superscheme} is a 
representable functor $X=\hom_{\salg_\k}(\k[X], -)$, for some $\k[X]\in \salg_\k$ 
which is called the {\em coordinate ring} of $X$.

Given $M \in \svec_\k$, let $M_a: \salg_\k \rightarrow \mathfrak{sets}$ 
denote the functor defined by 
$$M_a(\A)\ =\ (M\otimes \A)_0$$
for all $\A \in \salg_\k$.  Then $M_a$ is an affine super scheme 
with coordinate ring given 
as follows.  Suppose $N$ is an arbitrary superspace, not necessarily 
finite dimensional.  Then we may identify $M^\vee\otimes N$ with 
$\Hom_\k(M,N)$ by setting 
$$ (f\otimes w)(v)\ =\ (-1)^{|w| |v|} \la f, v \ra w \qquad (v\in M, w\in N, f\in M^\vee).$$
Then, for any $\A \in \salg_\k$, we have
\begin{align*}
M_a(\A) &=\ (M\otimes \A)_0 =\  ((M^\vee)^\vee \otimes \A)_0 \\
&=\  \Hom_\k(M^\vee, \A)_0\  =\  \hom_{\salg_\k}(S^*(M^\vee), \A).
\end{align*}
Hence $M_a$ is an affine superscheme with $\k[M_a] = S^*(M^\vee)$.

Now suppose $\mathcal{B}$ is an associative superalgebra.  Let 
$V,W\in {}_{\mathcal{B}}\smod$ and $\A \in \salg_\k$.  Then it 
may be checked that formula (\ref{eq:outer}) gives the following 
isomorphisms:
\begin{equation}\label{eq:tensor1}
\Hom_{\B}(V,W)\otimes \A\ \simeq\ \Hom_{{B}\otimes \A}(V\otimes \A, W\otimes \A),
\end{equation}
where $\A$ is viewed as a supermodule over itself with respect to left multiplication.

Let $End_{\B,V}$ denote the functor in $\Fct(\salg_\k,\mathfrak{sets})$
 such that $End_{\B,V}(\A)$ consists of the even $\B\otimes\A$-linear 
 endomorphisms from $V\otimes \A$ to itself.  Then, by identifying the 
 left and right hand sides of (\ref{eq:tensor1}), we see that 
 $End_{\B,V}= (\End_\B(V))_a$.  So that $End_{\B,V}$ is an affine 
 superscheme with $$\k[End_{\B,V}] = S^*(\End_\B(V)^\vee).$$  
Since $\End_\B(V)$ is a superalgebra, we may regard 
$\k[End_{\B,V}]$ as a cosuperalgebra via the map 
$\widetilde{\Delta}$ described above.

\subsection{Supergroups}
A {\em supergroup} is defined to be a functor $G$ from the category 
$\salg_\k$ to the category $\mathfrak{groups}$.  An {\em algebraic 
supergroup} is a supergroup $G$ which is also an affine superscheme, 
when viewed as a functor from $\salg_\k$ to $\mathfrak{sets}$, such 
that the coordinate ring $\k[G]$ is finitely generated.  In this case, $\k[G]$ 
has a canonical structure of  Hopf superalgebra.  In particular, the 
comultiplication $\Delta: \k[G] \rightarrow \k[G]\otimes \k[G]$ and counit 
$E: \k[G] \rightarrow G$ are defined, respectively, as the comorphisms 
of the multiplication and the unit of $G$.  

Suppose $\B$ is an associative superalgebra and $V\in {}_\B\smod$.
Let $GL_{\B,V}$ denote the subfunctor of $End_{\B,V}$ such that 
$GL_{\B,V}(\A)$ is the set of all even $\B\otimes \A$-linear 
{\em automorphisms} of $V\otimes \A$.  Then $GL_{\B,V}$ is an 
algebraic supergroup, and $\k[End_{\B,V}] = S^*(\End_\B(V)^\vee)$
is a subcoalgebra of $\k[GL_{\B,V}]$ with respect to the comultiplication
$\widetilde{\Delta}$ defined above.

\begin{example}
\begin{itemize}
\item[(i)]  Suppose $m,n$ are nonnegative integers. We use the notation 
$$Mat_{m|n}= End_{\k,\k^{m|n}}\quad  \mbox{and}\quad GL(m|n) = GL_{\k, \k^{m|n}}.$$
If $\A \in \salg_\k$, then 
$Mat_{m|n}(\A)$ may be identified with the set of all matrices of the form
\begin{equation}\label{eq:GL}\left( \begin{array}{cc}
A& B \\
C & D 
\end{array} \right), \end{equation}
where: $A$ is an $\A_0$-valued $m\times m$-matrix, $B$ is an $\A_1$-valued 
$m\times n$-matrix, $C$ is an $\A_1$-valued $n\times m$-matrix, and $D$ is 
an $\A_0$-valued $n\times n$-matrix.  The matrix (\ref{eq:GL}) corresponds to 
an even (resp. odd) linear operator if $B$ and $C$ (resp. $A$ and $D$) are 
both zero.  From \cite[Lemma 1.7.2 ]{Leites}, it follows that $GL(m|n,\A)$ 
consists of all matrices (\ref{eq:GL}) such that $\det(A)\det(D) \neq 0$.  

Let $M= \k^{m|n}$. If $f \in \End_\k(M)$, we may decompose $f = f_0 + f_1$, 
where $f_0$ is even and $f_1$ is odd. Let $\det \in S^{m+n}(\End_\k(M)^\vee)$ 
denote the element such that: for all $f\in \End_\k(M)$, $\det(f) = \det(\bar{f_0})$, 
where the latter is the usual determinant of the induced linear operator 
$\bar{f_0}:\overline{M} \rightarrow \overline{M}$ of ordinary vector spaces.
Then $GL(m|n)$ is an affine subsuperscheme of $Mat_{m|n}$, and $\k[GL(m|n)]$ 
is the localization of the coordinate ring $\k[Mat_{m|n}]$ at the element $\det$.
\medskip

\item[(ii)]  Suppose $n$ is a nonnegative integer, and let 
$V=\U_l(1)^{\oplus n} \in {}_{\C(1)}\smod$.  Then we write 
$$Mat_n = End_{\C(1),V} \quad \mbox{and} \quad Q(n)= GL_{\C(1),V}.$$
From Example \ref{ex:cliff2}, it follows that $Mat(\A)$ may be identified 
with the set of matrices of the form
\begin{equation}\label{eq:Q}\left( \begin{array}{cc}
S& S' \\
-S' & S 
\end{array} \right), \end{equation}
where $S$ (resp. $S'$) is an $\A_0$-valued (resp. $\A_1$-valued) $n\times n$-matrix.  
The matrix (\ref{eq:Q}) corresponds to an even (resp. odd) linear operator if $S' = 0$ 
(resp. $S=0$).

Then $Q(n,\A)$ consists of all 
{\em invertible} matrices of the form (\ref{eq:Q}).  We may define an 
element $\det \in \k[Mat] = S^*(\End_{\C(1)}(V)^\vee)$ in a 
way analogous to the previous example.
It follows from \cite{BKmodular} that $\k[Q(n)]$ is the localization of 
$\k[Mat]$ at $\det$.
\end{itemize}
\end{example}

A {\em representation} of an algebraic supergroup $G$ is defined to be a 
natural transformation $\eta: G \rightarrow GL_{\k,M}$ for some 
$M \in \svec_\k$ such that $\eta_\A: G(\A) \rightarrow  GL_{\k,M}(\A)$ is a 
group homomorphism for each $\A\in \salg_\k$.  On the other hand, a 
{\em $G$-supermodule}  is defined to be a right cosupermodule for the 
Hopf superalgebra $\k[G]$.  The two notions of supermodule and 
representation are equivalent (cf. \cite{BKmodular}).  In particular, given a 
representation $\eta:G \rightarrow GL_{\B,M}$, there is a corresponding 
{\em structure map}
$$\Delta_M: M \rightarrow M \otimes \k[G],$$
making $M$ into a $G$-supermodule.

\begin{definition}\label{def:polcat}
Suppose $\B$ is a superalgebra and $V\in {}_\B\smod$.  Let $G=GL_{\B,V}$. Then
we say that a representation $\eta: G \rightarrow GL_{\k,M}$ is {\em polynomial} if 
the image of the structure map $\Delta_M$ lies in $M\otimes \k[End_{\B,V}]$. 
We also let  $\pol_d(G)$ denote the category of all {\em homogeneous} polynomial 
representations of degree $d$, which are defined to be the representations $M$  
such that the image of $\Delta_M$ is contained in 
$\k[End_{\B,V}]_d = S^d(\End_\B(V)^\vee)$.  Notice that we have
$$\pol_d(G)\ =\ \cosmod_{\k[End_{\B,V}]_d}$$
since $S^d(\End_\B(V)^\vee)$ is a subcoalgebra of $S^*(\End_\B(V)^\vee)$.
\end{definition}

\medskip


\end{document}